\documentclass[12pt]{article}
\usepackage{amssymb,amsmath, enumitem, mathrsfs}
\usepackage[english]{babel}
\usepackage[utf8]{inputenc}
\usepackage[T1]{fontenc}
\usepackage{amsbsy}
\usepackage{rotating}
\usepackage{amsthm}
\usepackage{multirow}
\usepackage{wasysym}
\usepackage{mathrsfs}
\usepackage{verbatim}
\usepackage[colorlinks]{hyperref}
\usepackage{fullpage}
\usepackage{mathdots}
\usepackage{graphicx,subfigure}
\usepackage{stackengine}
\usepackage{tikz,tikz-cd}
\usepackage{tkz-graph}
\usepackage{booktabs}
\usepackage{mathtools}
\usepackage{authblk}
\usepackage{todonotes}
\usetikzlibrary{arrows,positioning,automata}
\PassOptionsToPackage{usernames,dvipsnames,svgnames}{xcolor}
\usepackage{pdflscape}
\usepackage{array}
\newcolumntype{L}{>{$}l<{$}}
\DeclareUnicodeCharacter{14C}{\=O}

\theoremstyle{definition}

\newtheorem{ndefn}{Definition}[section]

\newtheorem{nexap}[ndefn]{Example}

\theoremstyle{plain}

\newtheorem{nthm}[ndefn]{Theorem}

\newtheorem{nprop}[ndefn]{Proposition}

\newtheorem{nlem}[ndefn]{Lemma}

\newtheorem{ncor}[ndefn]{Corollary}

\theoremstyle{remark}

\newtheorem{nrk}[ndefn]{Remark}

\DeclarePairedDelimiter\floor{\lfloor}{\rfloor}
\DeclarePairedDelimiter\ab{\lvert}{\rvert}%
\DeclarePairedDelimiter\nm{\lVert}{\rVert}%

\makeatletter
\let\oldabs\ab
\def\ab{\@ifstar{\oldabs}{\oldabs*}}
\let\oldnorm\nm
\def\nm{\@ifstar{\oldnorm}{\oldnorm*}}
\makeatother

\DeclareMathOperator{\cint}{int}
\DeclareMathOperator{\spn}{span}
\DeclareMathOperator{\rank}{rank}
\DeclareMathOperator{\diag}{diag}

\DeclareMathOperator{\Gr}{Gr}
\DeclareMathOperator{\OGr}{OGr}

\DeclareMathOperator{\Rev}{Rev}
\DeclareMathOperator{\Comp}{Comp}

\numberwithin{equation}{section}
\numberwithin{figure}{section}

\newcommand{\N}{\mathbb{N}}

\newcommand{\C}{\mathbb{C}}

\newcommand{\R}{\mathbb{R}}

\newcommand{\x}{\times}

\newcommand{\la}{\langle}
\newcommand{\ra}{\rangle}

\newcommand{\vfi}{\varphi}
\newcommand{\sus}{\subseteq}

\newcommand{\subs}{\subset}

\newcommand{\tl}{\tilde}
\newcommand{\ol}{\overline}

\newcommand{\ep}{\varepsilon}
\newcommand{\del}{\Delta}
\newcommand{\lra}{\longrightarrow}

\newcommand{\llra}{\longleftrightarrow}

\newcommand{\xra}{\xrightarrow}
\newcommand{\lmt}{\longmapsto}
\newcommand{\wo}{\setminus}

\newcommand{\mc}[1]{\mathcal{#1}}

\newcommand{\mb}[1]{\mathbf{#1}}
\newcommand{\bb}[1]{\mathbb{#1}}

\newcommand{\ts}{\textsuperscript}

\newcommand{\lt}{\left}
\newcommand{\rt}{\right}

\newcommand{\arxiv}[1]{\href{http://arxiv.org/abs/#1}{\texttt{arXiv:#1}}}

\begin{document}

\title{Sects and lattice paths over the Lagrangian Grassmannian}

\author[1]{Aram Bingham}
\author[2]{\"Ozlem U\u{g}urlu}

\affil[1]{{\small Tulane University, New Orleans; abingham@tulane.edu}}    
\affil[2]{{\small Palm Beach State College, Boca Raton; ugurluo@palmbeachstate.edu}}

\normalsize

\date{\today}
\maketitle

\begin{abstract}
	
We examine Borel subgroup orbits in the classical symmetric space of type $CI$, which are parametrized by skew symmetric $(n,n)$-clans. We describe bijections between such clans, certain weighted lattice paths, and pattern-avoiding signed involutions, and we give a cell decomposition of the symmetric space in terms of collections of clans called sects. The largest sect with a conjectural closure order is isomorphic (as a poset) to the Bruhat order on partial involutions.

\vspace{.2cm}


\end{abstract}

\section{Introduction}\label{S:Introduction}
Let $G$ be simple algebraic group of classical type ($SL_n$, $SO_n$, or $Sp_{2n}$) over the complex numbers, and $\theta$ an automorphism of $G$ of order two. Then we call the fixed point subgroup $L:=G^\theta$ a \emph{symmetric subgroup} and $G/L$ a \emph{symmetric space} of classical type. If $B$ is a Borel subgroup of $G$, then $B$ acts on $G/L$ with finitely many orbits (\cite{matsukiOrbits}). The study of Borel orbits and their closures in symmetric spaces imitates and generalizes the study of Borel orbits in flag varieties, bearing comparable combinatorial richness. However, for only three types of classical symmetric spaces, $L$ happens to be a Levi subgroup of a (maximal) parabolic subgroup $P$; these are listed in Table~\ref{Table1} below. This makes it possible to relate the geometry and combinatorics of $B$-orbits in $G/L$ to those in $G/P$, via the ($B$-equivariant) canonical projection map, $\pi:G/L \to G/P$. 

In each of these cases, the homogeneous space $G/P$ parametrizes vector subspaces of $\C^n$ or $\C^{2n}$ which are isotropic with respect to a particular bilinear form, and is often called an \emph{(isotropic) Grassmannian} manifold/variety.  Indeed, the relevant symmetric spaces are those which are associated to a polarization of the appropriate vector space; see \cite{goodmanWallach} \S 11.3.5. The $B$-orbits of a Grassmannian are called \emph{Schubert cells}, as they are known to give a cell decomposition and an additive basis for (co)homology of the space $G/P$. Schubert cells can be parametrized by certain lattice paths which are also a tool for understanding their geometry.
\begin{table}[htp]\label{Table1}
	\centering
\begin{tabular}{@{}cccc@{}} \toprule
Type & Symmetric Pair & $B$-orbits parametrized by & $G/P$ \\
\midrule
$AIII$ & ($SL_{p+q}$, $S(GL_p \x GL_q))$ & $(p,q)$-clans & $\Gr(p, \C^{p+q})$ \\
$CI$ &( $Sp_{2n}$, $GL_n$ )& skew-symmetric $(n,n)$-clans & $\Lambda(n)$ \\
$DIII$ &( $SO_{2n}$, $GL_n$) & ``type $DIII$'' $(n,n)$-clans & $\OGr(n, \C^{2n})$ \\
\bottomrule
\end{tabular}
\caption{Classical symmetric subgroups which are also spherical Levi subgroups.}
\end{table}

$B$-orbits in $G/L$ are parametrized by objects dubbed \emph{clans} in \cite{matsukiOshima}, which have morphed in their development through subsequent works, notably \cite{yamamotoOrbits}, \cite{wyserThesis}, and \cite{canGenesis}. In \cite{bcSects}, it was shown that $\pi$ gives $G/L$ the structure of an affine bundle over $G/P$, and that the pre-images of Schubert cells provide a cell decomposition of $G/L$. This is used to conclude that the integral Chow rings and cohomology rings of $G/L$ are isomorphic. The type $AIII$ case is also treated in detail there, where the pre-images of Schubert cells in the Grassmannian of $p$-planes in $\C^{p+q}$, denoted $\Gr(p, \C^{p+q})$, are comprised of collections of \hbox{$(p,q)$-clans} called \emph{sects}. Each sect contains a unique closed $B$-orbit and a unique dense \hbox{$B$-orbit.} The closures of the dense \hbox{$B$-orbits} of each sect form a generating set for the integral Chow ring $A^*(G/L)$, akin to Schubert varieties. 

In this paper, we apply the ideas of \cite{canGenesis} and \cite{bcSects} to the symmetric space of type $CI$, wherein the ambient group is the symplectic group $Sp_{2n}$ and the symmetric subgroup is isomorphic to $GL_n$. Our first result is Theorem \ref{thm:restricted}, which counts Borel orbits in these symmetric spaces by providing a bijection between the parametrizing set of ``skew-symmetric'' \hbox{$(n,n)$-clans} and a set of pattern-avoiding signed involutions with known generating function. In Section \ref{C5:S4}, we describe another bijection of skew symmetric $(n,n)$-clans with a certain class of \emph{weighted $(n,n)$ Delannoy paths}. These are lattice paths in the plane from the origin to the point $(n,n)$, consisting only of north, east, and northeast diagonal steps, where the diagonal steps can have certain whole number weights. 

In Section \ref{sec:sects}, we describe the sects over the Schubert cells of the Lagrangian Grassmannian $\Lambda(n)$, which is the moduli space of maximal isotropic subspaces of the vector space $\C^{2n}$ with symplectic form $\Omega$. This proceeds in a fashion similar to the type $AIII$ case described in \cite{bcSects}, where it was also shown that the pre-image of the dense Schubert cell, called the \emph{big sect}, is isomorphic as a poset to the \emph{rook monoid} $\mc{R}_n$ with the Bruhat-Chevalley-Renner order. That result relied on a combinatorial description, given by Wyser in \cite{wyser2016bruhat}, of the closure order on $B$-orbits in the type $AIII$ symmetric space. 

For a classical symmetric space whose $B$-orbits are parametrized by a certain family of clans, one can describe the closure poset of clans by the order relation 
\[\gamma \leq \tau \iff \ol{Q_\gamma} \sus \ol{Q_\tau}\]
for clans $\gamma$, $\tau$, with corresponding $B$-orbits $Q_\gamma$, $Q_\tau$.  If $X$ is a classical symmetric space of type $B$, $C$, or $D$, then $X$ embeds in some symmetric space $X'$ of type $AIII$. The clans parametrizing $B$-orbits in $X$ can then be viewed as a subset of the $(p,q)$-clans parametrizing $B'$-orbits in $X'$, where $B'$ is a Borel subgroup of $SL_{p+q}$ and $B$ is the intersection of $B'$ and the relevant symplectic or special orthogonal subgroup. This reflects the fact that a $B$-orbit in $X$ indexed by a clan $\gamma$ is exactly the intersection of $X$ with the $B'$-orbit of $X'$ corresponding to the same clan (\cite{wyserThesis}, Theorem 1.5.8).

From this, one could hope that the closure order on clans in a type $B$-$C$-$D$ symmetric space would simply be the restriction of the relevant type $AIII$ closure order. This has been conjectured to be the case in types $BI$, $CI$, and $CII$, but it is known to fail in types $DI$ and $DIII$ (\cite{wyserKorbit}, \S 3.2.2). For this, among other reasons, the analysis for the type $DIII$ symmetric space warrants separate treatment.  Nevertheless, there is a \emph{weak order} on clans of a given type, whose order relations are contained within the full closure order, and which can be used to recover the closure order through a simple recursive procedure (see \cite{mcgovern2009pattern}). But the procedure does not appear to easily prove the conjectural closure order for type $CI$, so we proceed without it.

Our main result of Section~\ref{sec:big} says that the big sect over $\Lambda(n)$ with the conjectural closure order is isomorphic, as a poset, to the partial involutions on $n$ letters $\mc{P}_n$ with the \emph{Bruhat order} of \cite{bagnocongruence}. The order relations of the latter poset are given by the closure order on congruence orbits of upper triangular matrices acting on symmetric matrices, and they have a convenient combinatorial description which is provided below. A geometric argument explaining this coincidence and verifying that the closure relations within the big sect are indeed those of the congruence action will appear in the first author's Ph.D. thesis, along with analogous arguments for the two types. In type $DIII$, the orthogonal Grassmannian $\OGr(n,\C^{2n})$ of maximal isotropic subspaces with respect to a non-degenerate, symmetric, bilinear form appears as the base space of the bundle $\pi:G/L \to G/P$. Combinatorial analysis and description of the sects for this case will appear in forthcoming work from the authors.

\section{Notation and Preliminaries}\label{C5:S0}

All matrix groups in this paper are taken to have entries in the field of complex numbers. Let $n$ be a positive integer.
First, we must describe our realization of the type $CI$ symmetric pair $(Sp_{2n}, GL_n)$, borrowing notation from \cite{wyserThesis}. Let $J_n$ denote an $n \x n$ matrix with 1's along the anti-diagonal and 0's elsewhere. Let 
\[\Omega = \begin{pmatrix}
0 & J_n \\
-J_n & 0 
\end{pmatrix}.
\]
Then we set 
\begin{equation} \label{eq:Sp}
G := Sp_{2n} = \{ g\in GL_{2n} \mid  g^t \Omega g = \Omega \}. 
\end{equation}

Let $\cint(g): GL_{2n} \to GL_{2n}$ denote the map defined by 
\[ \cint (g)(h)= ghg^{-1} .\]
Now define the matrix 
\[I_{n,n} := \begin{pmatrix}
I_n & 0 \\
0 & -I_n 
\end{pmatrix}, \]
where $I_n$ denotes the $n\x n$ identity matrix.  Then we have an automorphism $\theta$ of order two on $GL_{2n}$ defined by $\theta:=\cint(iI_{n,n})$.
Indeed, \mbox{$(iI_{n,n})^{-1}=-i I_{n,n}$,} so if 
\[g=\begin{pmatrix}
A & B\\ 
C & D
\end{pmatrix}\]
is the $n \x n $ block form of $g$, we have
\[\theta(g)=
\begin{pmatrix}
i I_n & 0 \\ 
0 & -i I_n
\end{pmatrix}
\begin{pmatrix}
A & B \\
C & D 
\end{pmatrix}
\begin{pmatrix}
-i I_n & 0 \\ 
0 & i I_n
\end{pmatrix}
= \begin{pmatrix}
A & -B \\
-C & D
\end{pmatrix}. \]

Observe that the restriction of $\theta$ to $Sp_{2n}$ induces an order two automorphism on this group as well, since \mbox{$iI_{n,n}\in Sp_{2n}$.} The fixed points of $\theta$ must be block diagonal, that is
\[ \theta(g)=g \iff g= \begin{pmatrix}
A & 0 \\
0 & D
\end{pmatrix},\]
while membership in a symplectic group also forces $D=J_n(A^{-1})^tJ_n$. Thus, $A$ can be any invertible $n \x n$ matrix, and this completely determines $g$, so the fixed point subgroup $L:=G^\theta$ is isomorphic to $GL_n$, giving a type $CI$ symmetric pair.

Now, let us give a brief description of our involution notation. The symmetric group of permutations on $[n]:=\{1,\dots, n\}$ is denoted by $\mc{S}_n$. 
For instance, $\pi = (2, 6) (3, 4) (5, 7) (1) (8)$ is an example of a permutation from $\mc{S}_8$ which is written in cycle notation. 
If $\pi \in \mc{S}_n$, then its one-line notation 
is the string $\pi_1\pi_2\dots \pi_n$,
where $\pi_i = \pi(i)$ for $1\leq i \leq n$.
Based on this description, it is easy to see that the $\pi$ given above can be written as $\pi=16437258$ in one-line notation.

An \emph{involution} is an element of $\mc{S}_n$ of order at most $2$, and the set of involutions in $\mc{S}_n$ is denoted by $\mc{I}_n$. Let $\pi \in \mc{I}_n$ be an involution. Since we often need
the data of fixed points (1-cycles) of $\pi$, we always include them when writing $\pi$ in cycle notation. Thus, our standard form
for $\pi$ will be
$$
\pi = (a_1, b_1)(a_2, b_2) \dots (a_k, b_k) (d_1)\dots (d_{n-2k}),
$$ 
where $a_i < b_i$ for all $1 \leq i \leq k$, $a_1 < a_2 < \dots < a_k$, and $d_1<\dots < d_{n-2k}$. The example $\pi \in \mc{S}_8$ above is written in standard form.
\begin{ndefn} 
	A \emph{signed $(p,q)$-involution} is an involution $\pi\in \mc{I}_{p+q}$ with an assignment of $+$ and $-$ signs to the fixed points of $\pi$ 
	such that there are $p - q $ more $+$'s than $-$'s, where $q \leq p$. 
\end{ndefn}

For example, $\pi = (2, 6) (3, 4) (5, 7) (1^{-}) (8^{+})$ is a signed $(4, 4)$-involution. Observe here that $p$ is equal to the number of fixed points in $\pi$ with a $+$ sign attached plus the number of two-cycles in $\pi$, while $q$ is equal to the number of fixed points in $\pi$ with a $-$ sign attached plus the number of two-cycles in $\pi$. 
Next, we present $(p,q)$-clans.
\begin{ndefn}
	Let $p$ and $q$ be two positive integers and set $n:=p+q.$
	Suppose that $q \leq p$. A \emph{$(p,q)$-clan} $\gamma=c_1\cdots c_n$ is a string  
	of $n$ symbols from $\N\cup \{+,-\}$ such that
	\begin{enumerate}
		\item there are $p-q$ more $+$'s than $-$'s;
		\item if a natural number appears in $\gamma$, 
		then it appears exactly twice. 
	\end{enumerate} 
\end{ndefn}
For example, $12{+}21$ is a $(3,2)$-clan and ${+}1{+}1$ is a $(3,1)$-clan. We consider clans $\gamma$ and $\gamma'$ to be equivalent if the positions of every pair of matching numbers are the same in each.  For example, $\gamma:=1122$ and $\gamma':=2211$ are the same $(2,2)$-clan, since both $\gamma$ and $\gamma'$ have matching numbers in the positions $(1,2)$ and $(3,4)$. 

If $\gamma=c_1\cdots c_n$, then the \emph{reverse} of $\gamma$,
denoted by $rev(\gamma)$,
is the clan 
$$
rev(\gamma) = c_nc_{n-1}\cdots c_1.
$$
We take $-\gamma$ to be the clan obtained from $\gamma$ by changing all $+$'s to $-$'s, and vice versa.
\begin{ndefn} \label{ssClans}
	A $(p,q)$-clan $\gamma$ is called \emph{skew-symmetric} if 
	$$\gamma =  - rev(\gamma).$$
\end{ndefn} 
It is clear that skew-symmetric $(p,q)$-clans are only possible when $p=q$. Let us illustrate this definition with the following examples. 
\begin{nexap}
	Consider the clan $\gamma ={+}{-}123 3 1 2{+}{-}$ which has $rev(\gamma)= {-}{+}2 1 3 3 2 1{-}{+}$. Since $\gamma = -rev (\gamma)$, it is a skew-symmetric $(5, 5)$-clan. 
	
	The clan $\tau = 1 2 3 4 5 4 5 3 2 1$ is a skew-symmetric $(5,5)$-clan as well. In fact, since it has no $\pm$ symbols, $rev(\tau) = -rev(\tau)$. 
\end{nexap}
There is a one-to-one correspondence between $(p,q)$-clans and signed $(p,q)$-involutions; detailed proof can be found in~\cite{canGenesis}. Returning to our first example, the signed $(4,4)$-involution \hbox{$\pi = (2, 6) (3, 4) (5, 7) (1^{-}) (8^{+})$} can be regarded as the $(4, 4)$-clan ${-}122313{+}$. This is accomplished by placing matching natural numbers at the positions that appear in each transposition, and placing the signature $+$ or $-$ at the position of each signed fixed point. Observe here that $p$ is equal to the number of fixed points in $\pi$ with a $+$ sign attached plus the number of two-cycles in $\pi$, while $q$ is equal to the number of fixed points in $\pi$ with a $-$ sign attached plus the number of two-cycles in $\pi$. 
In the opposite direction, for example, the skew-symmetric $(5,5)$-clan $\tau=1 2 3 4 5 4 5 3 2 1$ becomes the signed $(5, 5)$-involution $(1, 10) (2,9)(3,8) (4, 6) (5, 7)$.

\section{Counting Skew-Symmetric $(n, n)$-Clans}\label{C5:S1}
If $B$ is a Borel subgroup of $Sp_{2n}$, then the $B$-orbits of the classical symmetric space $Sp_{2n}/GL_n$ are parameterized by skew-symmetric $(n,n)$-clans (\cite{yamamotoOrbits}, Theorem 3.2.11). Here, we obtain a formula for the number of Borel orbits in a type $CI$ symmetric space by counting skew-symmetric clans.

Let $\mc{Z}_n$ denote the set of all skew-symmetric $(n,n)$-clans and $Z_n$ denote its cardinality. Let $\zeta_{k,n}$ denote the number of such clans which contain $k$ pairs of natural numbers. In order to count $Z_n$, first we will count $\zeta_{k,n}$.


To determine an $(n,n)$-clan with $k$ pairs of natural numbers involves a placement of $\pm$ symbols in $2n-2k$ spots. By skew-symmetry, it is enough to focus on the first half of the string of length $2n$. Then, there are $n-k$ spots among the first half which can be $+$ or $-$, giving $2^{n-k}$ possibilities.
We have just proved the following lemma:
\begin{nlem} \label{pmlemma}
	There are $2^{n-k}$ ways of placing $\pm$ symbols among $2n-2k$ spots to obtain skew-symmetric $(n,n)$-clans with $k$ pairs of matching natural numbers.
\end{nlem}

\begin{nrk}\label{5:Rempm} 
Definition \ref{ssClans} implies that the number of $+$ signs among the first $n$ symbols of a given skew-symmetric $(n,n)$-clan must equal to the number of $-$ signs among the last $n$. 
\end{nrk}

\begin{nrk}\label{5:ineq}
The number of pairs of natural numbers in an $(n,n)$-clan is, of course, bounded between 0 and $n$.
\end{nrk}
Our next task is to determine the number of possible ways of placing $k$ pairs of natural numbers to build a skew-symmetric $(n,n)$-clan $\gamma=c_1 \cdots c_n c_{n+1}\cdots c_{2n}$. Together with the lemma, this will yield our main result for this section. 

\begin{nthm}\label{5:kcycles}
	For every nonnegative 
	integer $k$ with $k\leq n$, we have
	\begin{align}\label{5:ssclank}
	\zeta_{k,n} = 2^{n-k}{n \choose k} a_{k}, \;\;\;\;
	\text{where} \;\;\;\;
	a_{k} := \sum_{b=0}^{\floor{k/2}} {k \choose 2b} \frac{(2b)!}{b!},
	\end{align}
so that
	$$
	Z_n  = \sum_{k=0}^n \zeta_{k,n} = \sum_{k=0}^n  2^{n-k}{n \choose k} a_{k}.
	$$
\end{nthm}
\begin{proof}
	Let us first define the following interrelated sets:
	\begin{align*}
	\Pi_{1,1} :=& \{((i,j),(2n+1-j, 2n+1-i))\;|\; 1 \leq i < j \leq n \}, \\
	\Pi_{1,2} :=& \{((i,j),(2n+1-j, 2n+1-i))\;|\; 1 \leq i \leq n < j \leq 2n \}, \\
	\Pi_1 :=& \Pi_{1,1} \cup \Pi_{1,2},\\
	\Pi_2 :=& \{(i,j) \;|\; 1 \leq i \leq n < j \leq 2n, \; i+j = 2n+1 \}.
	\end{align*}
	Given a skew-symmetric $(n,n)$-clan $\gamma$, $\Pi_1$ is the set of possible placeholders for two distinct pairs of natural numbers that determine each other in $\gamma$. We will refer to these pairs of pairs as \emph{families}, after \cite{matsukiOshima}. 
	The set $\Pi_2$ corresponds to the list of possible
	``stand-alone'' pairs in $\gamma$. 

	If $(c_i,c_j)$ is a pair of matching natural numbers in the skew-symmetric clan $\gamma$
	and if $(i,j)$ is an element of $\Pi_2$, then we call $(c_i,c_j)$
	a pair of type $\Pi_2$. 
	If $((c_i,c_j),(c_{2n+1-j},c_{2n+1-i}))$ is a family in a skew-symmetric clan $\gamma$ and if \hbox{$((i,j),(2n+1-j,2n+1-i))\in \Pi_{1,s}$} ($s\in \{1,2\}$),
	then we call it a family of type $\Pi_{1,s}$.
	
	To illustrate these sets, consider the $(7,7)$-clan $$1{+}1 3 4{-}5 5{+}3 4 2{-}2.$$ For this clan,
	$((c_1,c_3), (c_{12}, c_{14}))$ is a type $\Pi_{1,1}$ family, $((c_4, c_{10}), (c_5, c_{11}))$ is a type $\Pi_{1,2}$ family, and $(c_7,c_8)$ is a $\Pi_2$ pair.

	Clearly, if a clan $\gamma$ has $ b$ many families and $a$ many $\Pi_2$ pairs, then $2b + a= k$ is 
	the total number of pairs in our skew-symmetric clan $\gamma$. 
	To see how many different ways there are in which these pairs of indices can 
	be situated in $\gamma$, we start by choosing $k$ spots
	from the first $n$ positions in $\gamma = c_1\cdots c_{2n}$. Obviously, this can be done in ${n\choose k}$ many different ways, and each choice made in the first half determines the second half of the clan uniquely.
	
	There are ${k\choose {2b}}$ possibilities for choosing which of these $k$ spots will be occupied by symbols from families. Then, we form $b$ pairs among these $2b$ elements which can be done in $\frac{(2b)!}{b!2^b}$ ways. But each of these pairs can either be in a $\Pi_{1,1}$ family or a $\Pi_{1,2}$, family, so we multiply by these $2^b$ additional options, resulting in a factor of $\frac{(2b)!}{b!}$. Observe that choosing these pairs is equivalent to choosing $(i,j)$ for the families in $\Pi_{1,1}$ and choosing $(i, 2n+1-j)$ for the families in $\Pi_{1,2}$. Once this is done, finally, the remaining $k-2b$ spots will be filled by the first symbols of the $a$ pairs of type $\Pi_2$. This can be done in only one way. 
	
Therefore, in summary, the number of different ways of placing $k$ pairs to build a skew-symmetric $(n,n)$-clan $\gamma$ is given by 
	\begin{align*}
		{n \choose k} \sum_{b=0}^{  \floor*{\frac{k}{2}}} {k \choose 2b} \frac{(2b)!}{b!},
	\hspace{.5cm} \text{ or equivalently, }  \hspace{.5cm}
	{n \choose k} \sum_{b=0}^{  \floor*{\frac{k}{2}}} {k \choose 2b}{2b \choose b} b! . 
	\end{align*}
Combining this with Lemma \ref{pmlemma} yields the formula for $\zeta_{k,n}$.
\end{proof}

\begin{nrk}
	By a straightforward calculation, it is easy to check that the recurrence
	\begin{align}\label{5:akrec}
	a_k = a_{k-1} + 2(k-1) a_{k-2},
	\end{align}
	 holds for all $k \geq 2$, with $a_0 = a_1 = 1$.
\end{nrk}

\begin{nexap}
	All the possible skew-symmetric $(3,3)$ clans are:
	\begin{center}
		${-} {-}{-}{+}{+}{+}, {+}{+}{+}{-} {-} {-}, {+} {+} {-} {+} {-} {-}, {+} {-} {+} {-} {+} {-}, {-} {+} {+} {-} {-} {+}, 
		{-}{-} {+} {-}{+} {+}$,\\
		${-} {+} {-} {+} {-} {+}, {+} {-} {-} {+} {+} {-}, 1 {+} {+} {-}{-} 1, {+}1{+} {-} 1{-}, {+} {+} 1 1 {-} {-}, 1 {-} {-} {+} {+}1$,\\
		${-} 1 {-} {+} 1 {+}, {-}{-}1 1{+} {+}, 1 {+}{-}{+}{-}1 , 1{-}{+} {-} {+}1, {-}1{+}{-}1 {+}, {+}1 {-}{+}1{-}$, \\ 
		${+}{-}11{+} {-}, {-}{+}11 {-}{+},  {+}1122 {-}, 1 1{-}{+}2 2, 1{-}12 {+} 2, {-}1122 {+}, 1 2 {+} {-} 1 2$, \\
		${-}1212{+},  12{-}{+} 12, 1{+}21 {-} 2, {+}1 21 2 {-}, 1{+}12{-} 2, 1 1 {+} {-}2 2, 1 2 {+} {-} 2 1$,\\
		$1{+}22{-}1, {+} 1 221{-}, 1 2{-} {+}21, 1{-}22{+} 1, {-} 122 1{+}, 1{-}21 {+} 2, 12 3321$, \\
		$12 31 23, 1 21 32 3, 1 2 32 31, 11 2 2 33, 123312, 122 3 31$.
	\end{center}
	Thus, there are $45$ skew-symmetric $(3,3)$ clans.
\end{nexap}

The first few values of $Z_n$ are $1, 3, 11, 45, 201, 963, 4899, 26253, 147345, 862083, 5238459, \dots$ This is also the number of signed involutions on $2n$ letters that are equal to their reverse-complements and avoid the pattern $\bar{2}\bar{1}$  (\cite{OEIS}, \href{http://oeis.org/A083886}{A083886}). We will refer such objects as \emph{restricted involutions} and explain the terminology used to define them in the next subsection.

\subsection{Restricted Involutions}

In this section, we will show that the elements of $\mc{Z}_n$ are in one-to-one correspondence with the set of restricted involutions $\hat{\mc{I}}_n$, which we introduce below. Our main reference for this section is~\cite{dukessigned}. 
Recall that the elements of the hyperoctahedral group $\mc{C}_n$ can be considered as signed permutations and
written as $\hat{\pi} = \hat{\pi}(1) \hat{\pi}(2) \dots \hat{\pi}(n)$, where each of the symbols $1, 2,\dots , n$ appears, possibly barred to indicate a negative.
For instance, there are 8 elements in $\mc{C}_2$, which are $12, 1 \bar{2}, \bar{1} 2, \bar{1} \bar{2}, 21, 2 \bar{1}, \bar{2} 1$, and $\bar{2} \bar{1}$, in one-line notation.

We define the absolute value of each symbol $\hat{\pi}(i)$ by
$$|\hat{\pi}(i)| = 
\begin{cases}
\hat{\pi}(i) \;\;\text{if} \;\; \hat{\pi}(i) \;\; \text{is positive,} \\
\overline{\hat{\pi}(i)} \;\; \text{otherwise},
\end{cases}$$
for any $1 \leq i \leq n$ and define the absolute value of a signed permutation $\hat{\pi}$ as the one obtained by taking the absolute value of each of its entries.
For example, if $\hat{\pi} = \ol{1} 2$ then $|\hat{\pi}|$ = 12. Indeed, since $\ol{\ol{1}} =1$, we have $|\hat{\pi}(1)|=1$. 

\begin{ndefn}
Let $\hat{\rho} \in \mc{C}_k$ and $\hat{\pi} \in \mc{C}_n$. We say that $\hat{\pi}$ contains a signed pattern $\hat{\rho}$, or is a \mbox{$\hat{\rho}$-containing}
signed permutation, if there is a sequence of $k$ indices, $1\leq i_1 < i_2 < \cdots < i_k \leq n$ such that
two conditions hold:
\begin{enumerate}
	\item $|\hat{\pi}({i_p})| > |\hat{\pi}({i_q})|$ if and only if $|\hat{\rho}(p)| > |\hat{\rho}(q)|$ for all $k \geq p > q \geq 1$;
	\item  $\hat{\pi}({i_j})$ is barred if and only if $\hat{\rho}(j)$ is barred for all $1 \leq j \leq k$.
\end{enumerate}
A signed permutation $\hat{\pi}$ which does not contain such a pattern $\hat{\rho}$ is said to \emph{avoid} $\hat{\rho}$.
\end{ndefn}

For example, $\hat{\pi} = 21 \bar{3} \bar{4} \in \mc{C}_4$ contains the signed pattern $\bar{1} \bar{2}$ but does not contain the pattern 12. 

\begin{ndefn}
A \emph{signed involution on $n$ letters} is an element of $\mc{C}_n$ such that its cycle
representation contains cycles of either two unbarred symbols or two barred symbols. 
\end{ndefn}

We define two simple operations on signed permutations. Given $\hat{\pi}= \hat{\pi}(1)\hat{\pi}(2) \dots \hat{\pi}(n)$ we consider the: 
\begin{itemize}
	\item Reverse permutation: $\Rev(\hat{\pi}$) = $\hat{\pi}(n) \dots \hat{\pi}(2) \hat{\pi}(1)$
	\item Complement permutation: $\Comp(\hat{\pi}$) $=\hat{\rho}(1) \hat{\rho}(2) \dots \hat{\rho}(n)$ where $\hat{\rho}(i) = n+1-\hat{\pi}(i)$ if $\hat{\pi}(i)>0$ and $-(n+1)+\hat{\pi}(i)$ otherwise.
\end{itemize}
These operations commute, allowing us to define our set of interest, called restricted involutions.
\begin{ndefn}
	A signed involution on $2n$ letters $\hat{\pi}$ is called a \emph{restricted involution} if it is equal to its reverse complement and avoids the pattern $\bar{2} \bar{1}$. The set of all such involutions for fixed $n$ will be denoted by $\hat{\mathcal{I}}_n$ and its cardinality by $\hat{I}_{n}$. 
\end{ndefn}
For example, $12, \bar{1} \bar{2}$, and $21$ are the restricted involutions of $\mc{C}_2$, which are the members of $\hat{\mc{I}}_1$.
We will prove that $\hat{{I}}_{n}= Z_n$ by exhibiting an explicit bijection between
skew-symmetric $(n, n)$-clans and the restricted involutions of $\mc{C}_{2n}$. This will allow us to make use of a
recurrence relation known for restricted involutions to give a generating function for the number of skew-symmetric $(n, n)$-clans and the orbits they parametrize.

First, we state the recurrence relation;
\begin{nprop}\label{resinvrec}
	Taking $\hat{{I}}_0 = 1$ and $\hat{{I}_1} =3$, the numbers $\hat{{I}}_n$ satisfy the following recurrence relation;
	\begin{center}
		$\hat{{I}}_n = 3\hat{{I}}_{n-1}+ 	2(n-1)\hat{{I}}_{n-2}$,
	\end{center}
	for all $n\geq 2$.
\end{nprop}

Let us make the notion of the \emph{underlying involution} of a clan precise. This is effectively the associated signed $(p,q)$-involution mentioned at the end of Section \ref{C5:S0}, but without the signs on the fixed points.
 \begin{ndefn}
	For a skew-symmetric $(n, n)$-clan $\gamma = c_1 c_2 \dots c_{2n}$, the associated underlying involution $\pi\in\mc{S}_{2n}$ is defined as follows:
\begin{enumerate}
	\item $\pi(i) = i$ if $c_i$ is either a $+$ or a $-$ for any $i$;
	\item  $\pi(i) = j$ and $\pi(j)= i$ if
	$c_i = c_j \in \N$ is a matching pair of natural numbers for any $i, j$.
\end{enumerate}
\end{ndefn}

The following describes an algorithm for obtaining a restricted involution $\hat{\pi}$ from a skew-symmetric clan $\gamma$ with underlying involution $\pi$ by possibly applying negatives to the symbols of $\pi$ in its one-line notation.
\begin{itemize}
	\item[(i)] If $c_i = +$ for $i \leq n$ (so, by skew-symmetry, $c_{2n+1-i} = -$), keep $$\hat{\pi}(i) = i\;\;\; \text{and}\;\;\; \hat{\pi}(2n+1-i) = 2n+1-i.$$
	\item[(ii)] If $c_i = -$ for $i \leq n$ (so, by skew-symmetry, $c_{2n+1-i} = +$), take $$\hat{\pi}(i) = \ol{i} \;\;\; \text{and}\;\;\; \hat{\pi}(2n+1-i) = \overline{2n+1-i}.$$
	\item[(iii)] If $c_i=c_j \in \N$ is a matching pair of natural numbers, then keep $$\hat{\pi}(i) = j\;\;\; \text{and} \;\;\; \hat{\pi}(j)= i$$ for all $i, j$. 
\end{itemize}

\begin{nexap}
Consider the skew-symmetric $(2, 2)$-clan $\gamma = {-} 1 1 {+}$ which has underlying involution $\pi=1324$ in one-line notation. Since $c_1 = -$ and $c_4 = +$, by rule (ii) we will have $\hat{\pi}(1) =\ol{1}$ and $\hat{\pi}(4) =\ol{4}$. Since $c_2 =c_3 = 1$, rule (iii) gives us $\hat{\pi}(2) =3$ and $\hat{\pi}(3) =2$. It follows that the associated restricted involution is $\ol{1} 3 2 \ol{4}$. 
\end{nexap}

Without trouble, this algorithm can be reversed to give a map from restricted involutions to clans, under which each restricted involution $\hat{\pi}$ will correspond to a unique clan $\gamma$. For a given $\hat{\pi}= \hat{\pi}(1) \hat{\pi}(2) \dots \hat{\pi}(2n)$, we define the reverse algorithm as follows.
\begin{itemize}
	\item[(i)]  If $\hat{\pi}(i) = i$ for any $1\leq i \leq n$, then $\hat{\pi}(2n+1-i) = 2n+1-i$ because $\hat{\pi} = \Rev(\Comp(\hat{\pi}))$. Then the clan $\gamma$ has $$c_i = +\;\;\; \text{and} \;\;\; c_{2n+1-i} = -.$$
	\item[(ii)]  If $\hat{\pi}(i) = \ol{i}$ for any $1\leq i \leq n$, then $\hat{\pi}(2n+1-i) = \overline{2n+1-i}$ because \mbox{$\hat{\pi} = \Rev(\Comp(\hat{\pi}))$}. Then the clan $\gamma$ has $$c_i = -\;\;\; \text{and}\;\;\; c_{2n+1-i} = +.$$
	\item[(iii)]  If $\hat{\pi}(i) = j$ for any $1\leq i <j \leq 2n$, then $$c_i=c_j=a,$$ where $a\in \N$. 
\end{itemize}
\begin{nrk}
Notice that the case $\hat{\pi}(i)=\ol{j}$ cannot occur in a restricted involution because this would force $\hat{\pi}(j)=\ol{i}$, producing a $\bar{2}\bar{1}$ pattern.
\end{nrk}
This algorithm and its reverse are clearly injective, so the fact that each restricted involution in $\hat{\mc{I}}_n$ gives a skew-symmetric $(n,n)$-clan completes the bijection.

\begin{nthm} \label{thm:restricted} Restricted involutions on $2n$ letters and skew-symmetric $(n,n)$-clans are in bijection.
\end{nthm}

\begin{ncor}
	Taking $Z_0=1$ and $Z_1=3$, the number of skew-symmetric $(n,n)$-clans satisfies the recurrence relation
	\begin{equation} \label{skewrec}
	Z_n = 3Z_{n-1} + 2(n-1)Z_{n-2},
	\end{equation}
	and has exponential generating function 
	\begin{equation}
	\sum_{n=0}^\infty Z_n\frac{x^n}{n!} = e^{3x+x^2}.
	\end{equation}
\end{ncor}
\begin{proof} The recurrence relation appears for restricted involutions as Theorem 4.2 in \cite{hardtTroyka}. The exponential generating function appears in the associated OEIS entry \href{http://oeis.org/A083886}{A083886}.
\end{proof}

\newpage
\subsection{Partial Orders on Skew-Symmetric Clans}
In this subsection, we depict two important partial orders on $\mc{Z}_n$, namely the \emph{weak} and \emph{full} (or \emph{Bruhat}) closure orders. We refer the reader to~\cite{wyserThesis} for explanation of the notation and a full combinatorial description of the weak order, which is lengthy.  The full closure order can be obtained from the weak order by applying the recursive procedure described in \cite{mcgovern2009pattern}, which is implicit in the work of Richardson and Springer \cite{richardsonSpringer1990}.
		\begin{figure}[htbp!]\label{weakorder2}
		\begin{center}
			\begin{tikzpicture}[scale=.3]
			\node at (-21,0) (a1) {${+}{+} {-} {-}$};
			\node at (-7,0) (a2) {${+} {-} {+} {-}$};
			\node at (7,0) (a3) {${-} {+} {-} {+}$};
			\node at (21,0) (a4) {${-} {-} {+} {+}$};

			\node at (-10,10) (b1) {${+} 1 1 {-}$};
			\node at (0,10) (b2) {$1 1 2 2$};
			\node at (10,10) (b3) {${-} 1 1 {+}$};

			\node at (-10,20) (c1) {$1 {+} {-} 1	$};
			\node at  (0,20) (c2) {$1 2 1 2$};
			\node at  (10,20) (c3) {$1 {-} {+} 1$};
			
			\node at (0,30) (d) {$1 2 2 1$};
			
			\path (a1) edge [thick]    node[left] {\tiny$2$}  (b1)
			(a2) edge [thick]    node[left] {\tiny$2$}  (b1)
			(a2) edge [thick]    node[right] {\tiny$1$}  (b2)
			(a3) edge [thick]    node[left] {\tiny$1$}  (b2)
			(a3) edge [thick]    node[right] {\tiny$2$}  (b3)
			(a4) edge [thick]    node[right] {\tiny$2$}  (b3)

			(b1) edge [thick]    node[left] {\tiny$1$}  (c1)
			(b2) edge [thick]    node[left] {\tiny$2$}  (c2)
			(b3) edge [thick]    node[right] {\tiny$1$}  (c3)
			
			(b1) edge [thick, dashed, red]    node[right] {}  (c2)
			(b2) edge [thick, dashed, red]    node[right] {}  (c1)
			(b2) edge [thick, dashed, red]    node[right] {}  (c3)
			(b3) edge [thick, dashed, red]    node[right] {}  (c2)

			(c1) edge [thick]    node[left] {\tiny$2$}  (d)
			(c2) edge [thick]    node[left] {\tiny$1$}  (d)
			(c3) edge [thick]    node[right] {\tiny$2$}  (d);
			\end{tikzpicture}
			\caption{Closure orders on $\mc{Z}_2$ }\label{5:Clans}
		\end{center}
	\end{figure}
	
In the posets, the black labelled edges are those which come from the weak order, while dashed edges are only present in the Bruhat order. 
\begin{landscape}
	\begin{figure}[htp]\label{weakorder3}
		\begin{center}
			\begin{tikzpicture}[scale=.25]
			\node at (-50,0) (a1) {${+}{+}{+}{-}{-}{-}$};
			\node at (-38,0) (a2) {${+}{+}{-}{+}{-}{-}$};
			\node at (-26,0) (a3) {${+}{-}{+}{-}{+}{-}$};
			\node at (-14,0) (a4) {${+}{-}{-}{+}{+}{-}$};
			\node at (-2,0) (a5) {${-}{+}{+}{-}{-}{+}$};
			\node at (10,0) (a6) {${-}{+}{-}{+}{-}{+}$};
			\node at (22,0) (a7) {${-}{-}{+}{-}{+}{+}$};
			\node at (34,0) (a8) {${-}{-}{-}{+}{+}{+}$};

			\node at (-50,10) (b1) {${+}{+}11{-}{-}$};
			\node at (-38,10) (b2) {${+}1122{-}$};
			\node at (-26,10) (b3) {${+}{-}11{+}{-}$};
			\node at (-14,10) (b4) {$11{+}{-}22$};
			\node at (-2,10) (b5) {$11{-}{+}22$};
			\node at (10,10) (b6) {$ {-}{+}11{-}{+}$};
			\node at (22,10) (b7) {${-}1122{+}$};
			\node at (34,10) (b8) {${-}{-}11{+}{+}$};

			\node at (-50,20) (c1) {${+}1{+}{-}1{-}$};
			\node at (-39,20) (c2) {${+}1212{-}$};
			\node at (-28,20)(c3) {${+}1{-}{+}1{-}$}; 
			\node at (-17,20) (c4) {$1{+}12{-}2$};
			\node at (-8,20) (c5) {$112233$};
			\node at (2,20) (c6) {$1{-}12{+}2$};
			\node at (12,20) (c7) {${-}1{+}{-}1{+}$};
			\node at (23,20) (c8) {${-}1212{+}$};
			\node at (34,20) (c9) {${-}1{-}{+}1{+}$};

			\node at (-50,30) (d1) {${+}1221{-}$};
			\node at (-39,30) (d2) {$1{+}{+}{-}{-}1$};
			\node at (-28,30) (d3) {$1{+}21{-}2$};
			\node at (-17,30) (d4) {$1{+}{-}{+}{-}1$};
			\node at (-8,30) (d5) {$121323$};
			\node at (2,30) (d6) {$1{-}{+}{-}{+}1$};
			\node at (12,30) (d7) {$1{-}21{+}2$};
			\node at (23,30) (d8) {$1{-}{-}{+}{+}1$};
			\node at (34,30) (d9) {${-}1221{+}$};

			\node at (-32,40) (e1) {$1{+}22{-}1$};
			\node at (-22,40) (e2) {$12{+}{-}12$};
			\node at (-12,40) (e3) {$122331$};
			\node at (-4,40) (e4) {$123123$};
			\node at (6,40) (e5) {$12{-}{+}12$};
			\node at (16,40) (e6) {$1{-}22{+}1$};

			\node at (-22,50) (f1) {$12{+}{-}21$};
			\node at (-12,50) (f2) {$123312$};
			\node at (-4,50) (f3) {$123231 $};
			\node at (6,50) (f4) {$12{-}{+}21$};
			
			\node at (-8,60) (g1) {$123321$};
			
			\path 
			(a1) edge [thick] node[left] {\tiny$3$} (b1)		
			(a2) edge [thick] node[left] {\tiny$3$}  (b1)
			(a2) edge [thick] node[left] {\tiny$2$}  (b2)
			(a3) edge [thick] node[left] {\tiny$2$}  (b2)
			(a3) edge [thick] node[left] {\tiny$3$}  (b3)
			(a3) edge [thick] node[right,pos=.2] {\tiny$1$}  (b4)
			(a4) edge [thick] node[left,pos=.2] {\tiny$3$}  (b3)
			(a4) edge [thick] node[right,pos=.2] {\tiny$1$}  (b5)
			(a5) edge [thick] node[left,pos=.2] {\tiny$1$}  (b4)
			(a5) edge [thick] node[right, pos=.2] {\tiny$3$}  (b6)
			(a6) edge [thick] node[left, pos=.2] {\tiny$1$}  (b5)
			(a6) edge [thick] node[left] {\tiny$3$}  (b6)
			(a6) edge [thick] node[left] {\tiny$2$}  (b7)
			(a7) edge [thick] node[left] {\tiny$2$}  (b7)
			(a7) edge [thick] node[left] {\tiny$3$}  (b8)
			(a8) edge [thick] node[right] {\tiny$3$}  (b8)

			(b1) edge [thick] node[left] {\tiny$2$}  (c1)
			(b2) edge [thick] node[left] {\tiny$3$}  (c2)
			(b2) edge [thick] node[left,pos=.2] {\tiny$1$}  (c4)
			(b3) edge [thick] node[right,pos=.15] {\tiny$2$}  (c3)
			(b3) edge [thick] node[left,pos=.2] {\tiny$1$}  (c5)
			(b4) edge [thick] node[left,pos=.15] {\tiny$2$}  (c4)
			(b4) edge [thick] node[left] {\tiny$3$}  (c5)
			(b5) edge [thick] node[left] {\tiny$3$}  (c5)
			(b5) edge [thick] node[right, pos=.15] {\tiny$2$}  (c6)
			(b6) edge [thick] node[right, post=.2] {\tiny$1$}  (c5)
			(b6) edge [thick] node[left,pos=.2] {\tiny$2$}  (c7)
			(b7) edge [thick] node[left,pos=.2] {\tiny$1$}  (c6)
			(b7) edge [thick] node[left] {\tiny$3$}  (c8)
			(b8) edge [thick] node[left] {\tiny$2$}  (c9)
			
			(c1) edge [thick] node[left] {\tiny$3$}  (d1)
			(c1) edge [thick] node[right,pos=.2] {\tiny$1$}  (d2)
			(c2) edge [thick] node[left,pos=.2] {\tiny$2$}  (d1)	
			(c2) edge [thick] node[left] {\tiny$1$}  (d3)
			(c3) edge [thick] node[left] {\tiny$3$}  (d1)
			(c3) edge [thick] node[right,pos=.2] {\tiny$1$}  (d4)
			(c4) edge [thick] node[left,pos=.2] {\tiny$3$}  (d3)
			(c5) edge [thick] node[left] {\tiny$2$}  (d5)
			(c6) edge [thick] node[right,pos=.2] {\tiny$3$}  (d7)
			(c7) edge [thick] node[left,pos=.2] {\tiny$1$}  (d6)
			(c7) edge [thick] node[left] {\tiny$2$}   (d9)
			(c8) edge [thick] node[left,pos=.2] {\tiny$1$}   (d7)
			(c8) edge [thick] node[left,pos=.2] {\tiny$2$}   (d9)
			(c9) edge [thick] node[right] {\tiny$1$}   (d8)
			(c9) edge [thick] node[right] {\tiny$3$}   (d9)
			
			(d1) edge [thick] node[left] {\tiny$3$}   (e1)
			(d2) edge [thick] node[left] {\tiny$1$}   (e1)
			(d3) edge [thick] node[right,pos=.2] {\tiny$2$}   (e2)
			(d4) edge [thick] node[left,pos=.2] {\tiny$3$}   (e1)
			(d4) edge [thick] node[left] {\tiny$2$}   (e3)
			(d5) edge [thick] node[right] {\tiny$1$}   (e3)
			(d5) edge [thick] node[right,pos=.2] {\tiny$3$}   (e4)
			(d6) edge [thick] node[left,pos=.2] {\tiny$2$}   (e3)
			(d6) edge [thick] node[right,pos=.2] {\tiny$3$}   (e6)
			(d7) edge [thick] node[left,pos=.2] {\tiny$2$}   (e5)
			(d8) edge [thick] node[right] {\tiny$3$}   (e6)
			(d9) edge [thick] node[right] {\tiny$1$}   (e6)

			(e1) edge [thick] node[left] {\tiny$2$}   (f1)
			(e2) edge [thick] node[left] {\tiny$1$}   (f1)
			(e2) edge [thick] node[left] {\tiny$3$}   (f2)
			(e3) edge [thick] node[right,pos=.2] {\tiny$3$}   (f3)
			(e4) edge [thick] node[left,pos=.2] {\tiny$2$}   (f2)
			(e4) edge [thick] node[right,pos=.2] {\tiny$1$}   (f3)
			(e5) edge [thick] node[left,pos=.2] {\tiny$3$}   (f2)
			(e5) edge [thick] node[right] {\tiny$1$}   (f4) 
			(e6) edge [thick] node[right] {\tiny$2$}   (f4)
			
			(f1) edge [thick] node[left] {\tiny$3$}  (g1)
			(f2) edge [thick] node[left] {\tiny$1$}  (g1)
			(f3) edge [thick] node[right] {\tiny$2$}  (g1)
			(f4) edge [thick] node[right] {\tiny$3$}  (g1)

			(b1) edge [thick, dashed, red]    node[right] {}  (c2)
			(b2) edge [thick, dashed, red]    node[right] {}  (c1)
			(b2) edge [thick, dashed, red]    node[right] {}  (c3)
			(b3) edge [thick, dashed, red]    node[right] {}  (c2)
			(b6) edge [thick, dashed, red]    node[right] {}  (c8)
			(b7) edge [thick, dashed, red]    node[right] {}  (c7)
			(b7) edge [thick, dashed, red]    node[right] {}  (c9)
			(b8) edge [thick, dashed, red]    node[right] {}  (c8)
			(c4) edge [thick, dashed, red]    node[right] {}  (d2)
			(c4) edge [thick, dashed, red]    node[right] {}  (d5)
			(c5) edge [thick, dashed, red]    node[right] {}  (d4)
			(c5) edge [thick, dashed, red]    node[right] {}  (d6)
			(c6) edge [thick, dashed, red]    node[right] {}  (d5)
			(c6) edge [thick, dashed, red]    node[right] {}  (d8)
			(d1) edge [thick, dashed,red]    node[right] {}  (e2)
			(d1) edge [thick, dashed, red]    node[right] {}  (e4)
			(d3) edge [thick, dashed, red]    node[right] {}  (e1)
			(d3) edge [thick, dashed, red]    node[right] {}  (e4)
			(d5) edge [thick, dashed, red]    node[right] {}  (e2)
			(d5) edge [thick, dashed, red]    node[right] {}  (e5)
			(d7) edge [thick, dashed, red]    node[right] {}  (e4)
			(d7) edge [thick, dashed, red]    node[right] {}  (e6)
			(d9) edge [thick, dashed, red]    node[right] {}  (e4)
			(d9) edge [thick, dashed, red]    node[right] {}  (e5)
			(e1) edge [thick, dashed, red]    node[right] {}  (f3)
			(e3) edge [thick, dashed, red]    node[right] {}  (f1)
			(e3) edge [thick, dashed, red]    node[right] {}  (f4)
			(e6) edge [thick, dashed, red]    node[right] {}  (f3);
			
			\end{tikzpicture}
			
			\caption{Closure orders on $\mc{Z}_3$ }
		\end{center}
	\end{figure}
\end{landscape}
\section{A combinatorial interpretation} \label{C5:S4}

In this section, we describe a combinatorial set of objects whose cardinality is given by $Z_n$. Recall that an \emph{$(n,n)$ Delannoy path} is an integer lattice path from $(0,0)$ to $(n,n)$ in the plane $\R^2$ consisting only of single north, east, diagonally northeast steps. Alternatively, one can consider strings from the alphabet $\{N, E, D \}$ such that the number of $N$'s plus the number of $D$'s is equal to the sum of the $E$'s and $D$'s (which is equal to $n$). We will denote the collection of such paths by $\mc{D}(n, n)$.

In \cite{canGenesis}, it was observed that the recurrence relation for the set of all $(p,q)$-clans bears strong resemblance to the recurrence relation for the set of Delannoy paths from $(0,0)$ to $(p,q)$. On the basis of this observation, a bijection between $(p,q)$-clans and Delannoy paths with certain weighted steps was established. Here, we provide a similar construction for skew-symmetric $(n,n)$-clans, which is modified to satisfy the appropriate recurrence (\ref{skewrec}) and so that the path associated to a closed orbit is the same as the path associated to the corresponding Schubert cell (see Section \ref{sasc}). Although the skew-symmetric $(n,n)$-clans form a subset of all $(n,n)$-clans, it is important to note that the paths produced here are not a subset of those constructed in \cite{canGenesis}. 

We produce an explicit bijection between the set of skew-symmetric \mbox{$(n, n)$-clans} $\mc{Z}_n$ and the set of Delannoy paths with certain labels which are defined as follows.

\begin{ndefn}
	By a \emph{labeled step} we mean a pair $(L,l)$,
	where $L\in \{N,E,D\}$ and $l$ is a positive integer
	such that $l=1$ if $L=N$ or $L=E$. 
	A \emph{weighted $(n,n)$ Delannoy path} is a word of the form 
	$W:=W_1\dots W_r$, where $W_i$'s $(i=1,\dots, r)$ 
	are labeled steps $W_i=(L_i,l_i)$ such that 
	\begin{itemize}
			\item $L_1\dots L_r$ is a Delannoy path from $\mc{D}(n,n)$.
			\item letting $W_{i_1}\dots W_{i_t}$ be the subword consisting of all weighted steps which are not $(D, 1)$, then $t$ is even. Further if $W_i=(D,1)$, then $i>i_{\frac{t}{2}}$.
			\item if $L_{i_j}=N$, then $L_{i_{t+1-j}}=E$ and vice versa for $1\leq j \leq t$.
			\item letting $$ d_i= \# \{ k>i \mid W_k=(D,1)\},$$
			    and  $$ m_i=\# \{k < i \mid l_k\neq 1\},$$ if $l_{i_j}\neq 1$ (so $L_{i_j}=D$), then $$W_{i_{t+1-j}}=(D,2n+3-2(i_j+m_{i_j}+d_{i_{t+1-j}})-(l_{i_j})),$$ for $1\leq j \leq \frac{t}{2}$.
	\end{itemize}
	The set of all weighted $(n,n)$ Delannoy 
	paths is denoted by $\mc{L}^w(n,n)$. The last condition, together with the fact that weights must be positive, implies that $1\leq l_i\leq 2n-1$ for any $i$.
\end{ndefn}

\begin{nthm}
	There is a bijection between the set of weighted $(n, n)$ Delannoy 
	paths and the set of skew-symmetric $(n,n)$-clans. 
	In particular, we have 
	$$
	Z_n = \sum_{W\in \mc{L}^{\omega}(n,n)} 1.
	$$
\end{nthm}
\begin{proof}
	Let $e_{n}$ denote the cardinality of $\mc{L}^{\omega}(n, n)$. 
	We will prove that the sequence of $e_{n}$'s 
	obeys the same recurrence as the $Z_n$'s, and it satisfies the 
	same initial conditions. 
	Let $\gamma = c_1\dots c_{2n}$ be an arbitrary skew-symmetric $(n, n)$-clan with associated signed $(n,n)$-involution
	$$
	\pi =(i_1, j_1) \dots (i_{k},j_{k}) (d_1^{s_1}) \dots (d_{2n-2k}^{s_{2n-2k}}),\
	\text{ where } s_1,\dots, s_{2n-2k}\in \{+,-\}.
	$$ 
	
	First, we look at the position of $2n$, which gives us four cases. 
	If $2n$ appears as a fixed point with a $+$ sign, $1$ must appear as a fixed point with a $-$ sign by the skew-symmetry condition. If this is the case, then we draw first an $E$-step between $(n-1, n)$ and $(n, n)$ and then an $N$-step between $(0, 0)$ and $(0, 1)$. We label both of these steps by 1 to turn them into labeled steps. Next, we remove the fixed points $1$ and $2n$ from $\pi$ and shift the remaining symbols down by 1 to obtain a signed $(n-1, n-1)$-involution corresponding to a skew-symmetric clan. There are $e_{n-1}$ possible ways of completing the drawn steps between $(0,1)$ and $(n-1,n)$ to an $(n, n)$ Delannoy path.
	
 In a similar manner, in case $2n$ appears as a fixed point with a $-$ sign, $1$ must appear as a fixed point with a $+$ sign. Then we draw an $N$-step between $(n, n-1)$ and $(n, n)$ and an $E$-step between $(0, 0)$ and $(1, 0)$, each with weight 1. Then there are again $e_{n-1}$ possible ways of completing these steps between $(1,0)$ and $(n,n-1)$ to a weighted $(n, n)$ Delannoy path.
	
	Next, consider the case where $2n$ appears in a transposition $(i, 2n)$ coming from a family in the clan $\gamma$. Then there is a companion 
	2-cycle, which is necessarily of the form $(1,j )$ for $j=2n+1-i$. So, we draw a $D$-step between $(n-1, n-1)$ and $(n, n)$ and
	label this step by $i$, and draw another $D$-step between $(0, 0)$ and $(1, 1)$ and
	label this step by $j$.
	Next we remove the two cycle $(i,2n)$ as well
	as its companion $(1,j)$ from $\pi$.
	To get rid of the gaps in the remaining numbers created by the removal of two 2-cycles, 
	we re-normalize the remaining entries by appropriately 
	subtracting numbers so that in the resulting object, 
	which we denote by $\pi^{(1)}$, every number from 
	$\{1,\dots, 2n-4\}$ appears exactly once. It is easy to 
	see that we then have a signed $(n-2, n-2)$-involution 
	which corresponds to a skew-symmetric $(n-2, n-2)$-clan. Now, it can occur that the label of either drawn diagonal step can be any of the	$2(n-1)$ numbers from $\{2,\dots, 2n-1\}$, but the choice of one specifies the other as they must add to $2n+1$. 
	Finally, let us note that there are $e_{n-2}$ 
	possible ways to complete these labeled diagonal steps between $(1,1)$ and $(n-1,n-1)$ to a weighted $(n,n)$ Delannoy path. 
	
	As the final case, consider when $2n$ appears in transposition with 1. Then, we draw a $D$-step between $(n, n)$ and $(n-1, n-1)$ and we 
	label this step by $1$. 
	Then we remove the 2-cycle $(1,2n)$ from $\pi$, and to get rid of the gap created, 
	we re-normalize the remaining entries by subtracting 1 from each so that in the resulting object $\pi^{(1)}$, every number from 
	$\{1,\dots, 2n-2\}$ appears exactly once.  
	Finally, let us note that there are $e_{n-1}$ 
	possible weighted paths from $(0,0)$ to $(n-1,n-1)$ to complete this labeled diagaonal step to a weighted $(n,n)$ Delannoy path. 
	
	Combining our observations we see that a weighted $(n, n)$ Delannoy path can be obtained (by appending the initial or final weighted steps described above) from an $(n-1,n-1)$-clan in 3 ways, or from an $(n-2,n-2)$-clan in $2n-2$ different ways. Thus, weighted $(n, n)$ Delannoy paths satisfy the recurrence
	\begin{align}\label{5:total}
	e_n=3e_{n-1} + 2(n-1) e_{n-2}.
	\end{align}
We can take $e_0=1$, and it is easy to check $e_1=3$, where the three paths are $(N,1)(E,1)$, $(E,1)(N,1)$ and $(D,1)$. An explicit bijection is achieved by iterating the steps above and removing entries from $\pi$ at each stage to produce weighted steps. This finishes our proof. 
\end{proof}

Let us illustrate our construction by an example.

\begin{nexap}
	Let $\gamma$ denote the skew-symmetric $(5,5)$-clan 
	$$
	\gamma = {+}{-}1 2 3 3 1 2{+}{-},
	$$
	and let $\pi$ denote the corresponding signed involution 
	$$
	\pi = (3,7)(4,8)(5,6)1^+ 2^- 9^+ 10^-.
	$$ 
	The steps of our constructions 
	are shown in Figure~\ref{5:last pic}.
	\begin{figure}[htb!]
		\begin{center}
			\begin{tikzpicture}[scale=.35]
			\begin{scope}[xshift = -5cm]
			\node at (0,2.5) {$\pi =(3,7)(4,8)(5,6)1^+ 2^-  9^+ 10^-$};
			\node at (0,.5) {$\pi^{(1)}= (2, 6) (3, 7) (4, 5)1^-  8^+$};
			\draw[ultra thick,red, ->] (0,2) to (0,1);
			\end{scope}
			\begin{scope}[xshift = 4.5cm]
			\draw (0,0) grid (5,5);
			\draw[ultra thick,red] (5,5) to (5,4);
			\draw[ultra thick,red] (0,0) to (1,0);
			\node[blue] at (5.5, 4.5) {$1$};
			\node[blue] at (0.5, -0.5) {$1$};
			\end{scope}
			\begin{scope}[xshift = -5cm,yshift=-7cm]
			\node at (0,2.5)  {$\pi^{(1)}= (2, 6) (3, 7) (4, 5)1^-  8^+$};
			\node at (0,.5) {$\pi^{(2)}= (1, 5) (2, 6) (3, 4)  $};
			\draw[ultra thick,red, ->] (0,2) to (0,1);
			\end{scope}
			\begin{scope}[xshift = 4.5cm,yshift=-7cm]
			\draw (0,0) grid (5,5);
			\draw[ultra thick,red] (5,5) to (5,4);
			\draw[ultra thick,red] (5,4) to (4,4);
			\draw[ultra thick,red] (0,0) to (1,0);
			\draw[ultra thick,red] (1,0) to (1,1);
			\node[blue] at (5.5, 4.5) {$1$};
			\node[blue] at (4.5, 3.5) {$1$};
			\node[blue] at (0.5, -0.5) {$1$};
			\node[blue] at (1.5, 0.5) {$1$};
			\end{scope}
			\begin{scope}[xshift = -5cm,yshift=-14cm]
			\node at (0,2.5) {$\pi^{(2)}= (1, 5) (2, 6) (3, 4)  $};
			\node at (0,.5) {$\pi^{(3)}= (1, 2)$};
			\draw[ultra thick,red, ->] (0,2) to (0,1);
			\end{scope}
			\begin{scope}[xshift = 4.5cm,yshift=-14cm]
			\draw (0,0) grid (5,5);
			\draw[ultra thick,red] (5,5) to (5,4);
			\draw[ultra thick,red] (5,4) to (4,4);
			\draw[ultra thick,red] (0,0) to (1,0);
			\draw[ultra thick,red] (1,0) to (1,1);
			\draw[ultra thick,red] (1,1) to (2,2);
			\draw[ultra thick,red] (3,3) to (4,4);
			\node[blue] at (5.5, 4.5) {$1$};
			\node[blue] at (4.5, 3.5) {$1$};
			\node[blue] at (0.5, -0.5) {$1$};
			\node[blue] at (1.5, 0.5) {$1$};
			\node[blue] at (3.5, 3.5) {$2$};
			\node[blue] at (1.5, 1.5) {$5$};
			\end{scope}

			\begin{scope}[xshift = -5cm,yshift=-21cm]
			\node at (0,2.5) {$\pi^{(3)}= (1, 2)$};
			\node at (0,.5) {$\pi^{(4)}= \cdot$};
			\draw[ultra thick,red, ->] (0,2) to (0,1);
			\end{scope}
			\begin{scope}[xshift = 4.5cm,yshift=-21cm]
			\draw (0,0) grid (5,5);
			\draw[ultra thick,red] (5,5) to (5,4);
			\draw[ultra thick,red] (5,4) to (4,4);
			\draw[ultra thick,red] (0,0) to (1,0);
			\draw[ultra thick,red] (1,0) to (1,1);
			\draw[ultra thick,red] (1,1) to (2,2);
			\draw[ultra thick,red] (3,3) to (4,4);
			\draw[ultra thick,red] (2,2) to (3,3);
			\node[blue] at (5.5, 4.5) {$1$};
			\node[blue] at (4.5, 3.5) {$1$};
			\node[blue] at (0.5, -0.5) {$1$};
			\node[blue] at (1.5, 0.5) {$1$};
			\node[blue] at (3.5, 3.5) {$2$};
			\node[blue] at (1.5, 1.5) {$5$};
			\node[blue] at (2.5, 2.5) {$1$};
			\end{scope}
			\end{tikzpicture}
		\end{center}
		\caption{Algorithmic construction of the bijection onto weighted Delannoy paths.}
		\label{5:last pic}
	\end{figure}
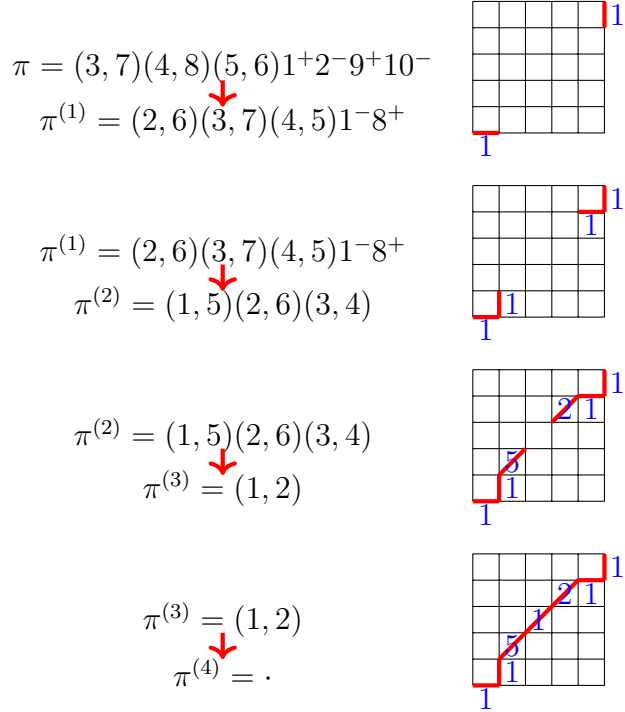
\end{nexap}

One could describe the weak and full closure orders on $B \backslash Sp_{2n}/GL_n$ in terms of weighted lattice paths with appropriate statistics. However, this rephrasing of the order on clans does not seem to be as illuminating as it is for the order on Schubert cells in the (Lagrangian) Grassmannian.  Lattice paths associated to Schubert cells determine partition shapes, whence their closure order coincides with the order relations from Young's lattice (given by containment of partition shapes).
\begin{figure}[htp]
	\begin{center}
		\begin{tikzpicture}[scale=.25]
		\node at (-15,0) (a1) {$\begin{tikzpicture}[scale=.5]
			\draw (0,0) grid (2,2);
			\draw[ultra thick,red] (0,0) to (2,0);
			\draw[ultra thick,red] (2,0) to (2,2);
			\end{tikzpicture}	$};
		\node at (-5,0) (a2) {$\begin{tikzpicture}[scale=.5]
			\draw (0,0) grid (2,2);
			\draw[ultra thick,red] (0,0) to (1,0);
			\draw[ultra thick,red] (1,0) to (1,1);
			\draw[ultra thick,red] (1,1) to (2,1);
			\draw[ultra thick,red] (2,1) to (2,2);
			\end{tikzpicture}	$};
		\node at (5,0) (a3) {$\begin{tikzpicture}[scale=.5]
			\draw (0,0) grid (2,2);
			\draw[ultra thick,red] (0,0) to (0,1);
			\draw[ultra thick,red] (0,1) to (1,1);
			\draw[ultra thick,red] (1,1) to (1,2);
			\draw[ultra thick,red] (1,2) to (2,2);
			\end{tikzpicture}	$};
		\node at (15,0) (a4) {$\begin{tikzpicture}[scale=.5]
			\draw (0,0) grid (2,2);
			\draw[ultra thick,red] (0,0) to (0,2);
			\draw[ultra thick,red] (0,2) to (2,2);
			\end{tikzpicture}	$};

		\node at (-10,10) (b1) {$\begin{tikzpicture}[scale=.5]
			\draw (0,0) grid (2,2);
			\draw[ultra thick,red] (0,0) to (1,0);
			\draw[ultra thick,red] (1,0) to (2,1);
			\draw[ultra thick,red] (2,1) to (2,2);
			\node[blue] at (1.5,0.5) {$1$};
			\end{tikzpicture}	$};
		\node at (0,10) (b2) {$\begin{tikzpicture}[scale=.5]
			\draw (0,0) grid (2,2);
			\draw[ultra thick,red] (0,0) to (2,2);
			\node[blue] at (1.5,1.5) {$3$};
			\node[blue] at (0.5,0.5) {$2$};
			\end{tikzpicture}	$};
		\node at (10,10) (b3) {$\begin{tikzpicture}[scale=.5]
			\draw (0,0) grid (2,2);
			\draw[ultra thick,red] (0,0) to (0,1);
			\draw[ultra thick,red] (0,1) to (1,2);
			\draw[ultra thick,red] (1,2) to (2,2);
			\node[blue] at (.5,1.5) {$1$};
			\end{tikzpicture}	$};

		\node at (-10,20) (c1) {$\begin{tikzpicture}[scale=.5]
			\draw (0,0) grid (2,2);
			\draw[ultra thick,red] (0,0) to (1,0);
			\draw[ultra thick,red] (1,0) to (1,1);
			\draw[ultra thick,red] (1,1) to (2,2);
			\node[blue] at (1.5,1.5) {$1$};
			\end{tikzpicture}	$};
		\node at  (0,20) (c2) {$\begin{tikzpicture}[scale=.5]
			\draw (0,0) grid (2,2);
			\draw[ultra thick,red] (0,0) to (2,2);
			\node[blue] at (1.5,1.5) {$2$};
			\node[blue] at (0.5,0.5) {$3$};
			\end{tikzpicture}	$};
		\node at  (10,20) (c3) {$\begin{tikzpicture}[scale=.5]
			\draw (0,0) grid (2,2);
			\draw[ultra thick,red] (0,0) to (0,1);
			\draw[ultra thick,red] (0,1) to (1,1);
			\draw[ultra thick,red] (1,1) to (2,2);
			\node[blue] at (1.5,1.5) {$1$};
			\end{tikzpicture}	$};
		
		\node at (0,30) (d) {$\begin{tikzpicture}[scale=.5]
			\draw (0,0) grid (2,2);
			\draw[ultra thick,red] (0,0) to (2,2);
			\node[blue] at (.5,.5) {$1$};
			\node[blue] at (1.5,1.5) {$1$};
			\end{tikzpicture}	$};
		
		\draw[-,  thick] (a1) to (b1);
		\draw[-,  thick] (a2) to (b1);
		\draw[-,  thick] (a2) to (b2);
		\draw[-,  thick] (a3) to (b2);
		\draw[-,  thick] (a3) to (b3);
		\draw[-,  thick] (a4) to (b3);
		
		\draw[-, thick] (b1) to (c1);
		\draw[-, thick] (b2) to (c2);
		\draw[-, thick] (b3) to (c3);

		\draw[-, thick] (c1) to (d);
		\draw[-, thick] (c2) to (d);
		\draw[-, thick] (c3) to (d);
		\end{tikzpicture}
		
		\caption{Weak order on $\mc{L}^{\omega}(2, 2)$ }\label{5:LPaths}
	\end{center}
\end{figure}

\section{Sects} \label{sec:sects}

In this section, we will verify that the framework of \cite{bcSects}, described in the introduction, applies to the $CI$ case. 

\subsection{Parabolic Subgroups and Levi Factors}

In order to present sects for skew-symmetric $(n, n)$-clans, we must visit the theory of parabolic subgroups of symplectic groups. We refer to \cite{malleTesterman} for background theory.

Given a vector space $V$ with bilinear form $\omega$, recall that an \emph{isotropic subspace} $W$ is one such that $\omega(\mb{u},\mb{v}) = 0$ for all vectors $\bf{u}, \bf{v} \in W$. If we also use $\omega$ to stand for the matrix which represents this bilinear form in some chosen basis, this condition becomes  $\bf{u}^t \omega \bf{v} =0$. A \emph{polarization} of $V$ is a direct sum decomposition of $V$ into subspaces which are each isotropic (with respect to $\omega$), that is
$V = V_- \oplus V_+$.

Given a vector space $V$ with bilinear form $\omega$, we define an \emph{isotropic flag} as a sequence of vector spaces
\[ \{\mathbf{0} \} \subs V_1 \subs \dots V_r \subs V \]
such that $V_i$ is an isotropic subspace of $V$ for all $1\leq i \leq r$. Taking $V=\C^{2n}$, we have a skew-symmetric bilinear form given by the same $\Omega$ which defines the symplectic group. From \cite{malleTesterman}, Proposition 12.13, the parabolic subgroups of $G=Sp_{2n}$ are precisely the stabilizers of flags which are isotropic with respect to $\Omega$.

Let $E_n$ be the subspace of $\C^{2n}$ generated by standard basis vectors $\mb{e}_i$, $1\leq i \leq n$. It is easy to check that this is an isotropic subspace of $\C^{2n}$ with respect to $\Omega$, and in fact this subspace is maximally isotropic, or \emph{Lagrangian}. 
The stabilizer of $E_n$ (which is the stabilizer of the the flag $\{\mb{0}\} \subs E_n \subs \C^{2n}$) is the parabolic subgroup $P$ consisting of matrices with $n \x n$ block form
\begin{equation}\label{eq:PL}
P= \lt\{p\in G\; \middle| \; p=  \begin{pmatrix}
* & * \\
0 & * 
\end{pmatrix}
\rt\}
,\;
\text{which has Levi factor } \;
L =\lt\{ \begin{pmatrix}
A & 0 \\
0 & J_n(A^{-1})^t J_n 
\end{pmatrix} \;\middle|\; A \in GL_n \rt\}. 
\end{equation}
See \cite{malleTesterman} p. 144 or \cite{garrettBuildings} \S 8.1 for related discussion. Thus, we see that the Levi subgroup of this parabolic subgroup coincides with the symmetric subgroup presented in Section \ref{C5:S0}.

There is also a polarization of $\C^{2n}$ as 
\[V=E_n\bigoplus W,\]
where $W$ is the subspace spanned by $\{\mb{e}_{n+1},\dots,\mb{e}_{2n}\}$. Notice that $L$ is exactly the stabilizer subgroup of this polarization, as it preserves each component. Since $G=Sp_{2n}$ acts transitively on polarizations, we can identify $G/L=Sp_{2n}/GL_n$ with the space of polarizations of the symplectic vector space $(\C^{2n},\Omega)$. 

The upshot of this is that we have a $G$-equivariant projection map 
\[\pi: G/L \lra G/P\]
which we can analyze. Let $B$ be the Borel subgroup of upper triangular matrices in $G$ (\cite{malleTesterman}, p. 39). The $B$-orbits in $G/P$ are collections of Lagrangian subspaces which form Schubert cells, while $B$-orbits in $G/L$ are collections of polarizations indexed by skew symmetric $(n,n)$-clans. The equivariance of $\pi$ allows us to ask precisely which clans constitute the pre-image of a particular Schubert cell. We call such a collection of clans the \emph{sect} associated to the Schubert cell.

In \cite{yamamotoOrbits}, clans parametrize $L$-orbits in the isotropic flag variety $G/B$ by encoding the information of how flags in a given orbit intersect with each component of the reference polarization $E_n\oplus W$.\footnote{See Definition  2.1.7 and Proposition 2.2.6 of the cited work for details.} Recall that a \emph{full} isotropic flag $F_\bullet$ in $\C^{2n}$ is a sequence of vector subspaces $\{V_i\}_{i=0}^n$ such that
\begin{equation}\label{eq:isoflag}
\{\mb{0}\}= V_0 \subs V_1 \subs V_2 \subs \dots \subs V_n 
\end{equation} 
and $\dim V_i=i$ for all $i$ and $V_n$ is a maximal isotropic subspace. We find it convenient to write 
\[ F_\bullet = \la \mb{v}_1, \mb{v}_2, \dots, \mb{v}_n  \ra \] to indicate that $F_\bullet$ is the flag with $V_i= \spn \{ \mb{v}_1, \dots, \mb{v}_i \}$ for all $1\leq i \leq n $. Any full isotropic flag is canonically extended to a full flag in $\C^{2n}$
\[\{\mb{0}\} \subs V_1 \subs \dots \subs V_{2n-1} \subs V_{2n}=\C^{2n} \]
by assigning 
\[V_{2n-i} = V_i^\perp := \{ \mb{v}\in \C^{2n} \mid \Omega(\mb{v},\mb{w})= 0,\  \forall \mb{w}\in V_i \}, \]
so we may abuse notation slightly by using $F_\bullet$ to refer to either presentation. For example, the standard isotropic full flag is extended as
\[ E_\bullet= \la \mb{e}_1, \dots, \mb{e}_n, \mb{e}_{n+1},\dots, \mb{e}_{2n} \ra.\]


If $g\in G$ is a matrix whose $i$\ts{th} column is a vector $\mb{v}_i$, then one can obtain a full isotropic flag from $g$ by taking $F_\bullet=\la \mb{v}_1, \dots, \mb{v}_{2n}\ra$. For example, the identity matrix $I_{2n}$ gives the standard isotropic flag $E_\bullet$. All elements of the same (right) $B$-orbit give the same flag, so one can identify the points of $G/B$ with full isotropic flags with respect to $\Omega$ in $\C^{2n}$.

We must present a few definitions before describing the process of obtaining orbit-representative flags; see also \cite{bcSects}.
\begin{ndefn}
	Given an $(n,n)$-clan $\gamma= c_1\cdots c_{2n}$, one obtains the \emph{default signed clan} associated to $\gamma$ by assigning to $c_i$ a signature of $-$ and to $c_j$ a signature of $+$ whenever $c_i=c_j \in \N$ and $i< j$. We denote this default signed clan as $\tl{\gamma} =\tl{c}_1 \cdots  \tl{c}_{2n}$. 
\end{ndefn}

\begin{nrk} Note that the definition above differs from Definition 3.2.8 of ``skew-symmetric signed clan'' given in \cite{yamamotoOrbits}, which makes the opposite choice. This will also affect our statement of Yamamoto's Theorem 3.2.11, given as Theorem \ref{thm:flag} below.
\end{nrk}
For example, $\tl{\gamma} = {+}1_- 2_- 1_+ 2_+{-}$ is the default signed clan of $\gamma={+} 1 2 1 2 {-}$. Every symbol $c_i$ has a signature, which is just the symbol itself in case $c_i$ is $+$ or $-$.
\begin{ndefn} Given a default signed clan $\tl{\gamma}$, define a permutation ${\sigma} \in \mc{S}_{2n}$ which, for $i\leq n$:
	\begin{itemize}
		\item assigns $\sigma(i)=i$  and $\sigma(2n+1-i)= 2n+1-i$ if $c_i$ is a symbol with signature $+$.
		\item assigns $\sigma(i)=2n+1-i$ and $\sigma(2n+1-i)=i$ if $c_i$ is a symbol with signature $-$.
	\end{itemize} We call ${\sigma}$ the \emph{default permutation} associated to $\gamma$. 
\end{ndefn}
Note that $\sigma$ is an involution, and is the $\sigma'$ which results from choosing $\sigma''=id$ in the context of \cite{yamamotoOrbits}, Theorem 3.2.11. For instance, ${+}1212{-}$ has default permutation $1 5 4 3 2 6$ in one-line notation.  

\subsection{Sects for Skew-Symmetric Clans}
For this subsection, fix $B$ as the Borel subgroup of $G=Sp_{2n}$ consisting of upper triangular matrices, and $P$ and $L$ as defined by the condition (\ref{eq:PL}). The following is a specialization of \cite{yamamotoOrbits} Theorem 3.2.11, which gives representative flags for $L$-orbits in the isotropic flag variety $G/B$.
\begin{nthm} \label {thm:flag}Given a skew-symmetric $(n,n)$-clan $\gamma=c_1 \cdots c_{2n}$ with default permutation ${\sigma}$, define a flag $F_\bullet(\gamma)=\la \mb{v}_1, \dots, \mb{v}_{2n}\ra$ by making the following assignments.
	
	\begin{itemize}
		\item If $c_i = - $, set \[\mb{v}_i = \mb{e}_{\sigma(i)} \]
		\item
		If $c_i= +$, set
		\[ \mb{v}_i = \begin{cases} \mb{e}_{\sigma(i)} &\textrm{  if  } i\leq n,\\
		-\mb{e}_{\sigma(i)} &\textrm{  if  } i> n .
		\end{cases} 
		\] 
		\item If $c_i=c_j \in \N$ where $\tl{c_i}$ has signature $-$ and $\tl{c_j}$ has signature $+$ (that is, $i< j$), then set 
		\[ \mb{v}_i = \begin{cases}
		\frac{1}{\sqrt{2}}(\mb{e}_{\sigma(i)}+\mb{e}_{\sigma(j)})  &\textrm{  if  } i\leq n,\\
		-\frac{1}{\sqrt{2}}(\mb{e}_{\sigma(i)}+\mb{e}_{\sigma(j)}) &\textrm{  if  } i> n .
		\end{cases}
		\]
		and
		\[ \mb{v}_j= \frac{1}{\sqrt{2}}(\mb{e_{\sigma(i)}}-\mb{e_{\sigma(j)}}).\]
	\end{itemize}
	Then $F_\bullet(\gamma)$ is a representative flag for the $L$-orbit $Q_\gamma$ in $G/B$. Furthermore, if $g_\gamma$ is the matrix defined by letting $\mb{v}_i$ be its $i$\ts{th} column, then $Q_\gamma=Lg_\gamma B/B$. Matrices/flags obtained in this way from skew-symmetric clans constitute a full set of representative flags for $L$-orbits in $G/B$.
\end{nthm}

For example, the matrix representative for the clan $\gamma={+}1212{-}$ is 
\[ g_\gamma =
\begin{pmatrix}
1 & 0 & 0 & 0 & 0 & 0 \\
0 & 0 & \frac{1}{\sqrt{2}} & 0 & -\frac{1}{\sqrt{2}} &0 \\
0 & \frac{1}{\sqrt{2}} & 0 & -\frac{1}{\sqrt{2}} & 0 & 0 \\
0 & 0 & \frac{1}{\sqrt{2}} & 0 &  \frac{1}{\sqrt{2}} & 0 \\
0 & \frac{1}{\sqrt{2}} & 0 & \frac{1}{\sqrt{2}} & 0 & 0 \\
0 & 0 & 0 & 0 & 0 & 1
\end{pmatrix}.
\]
 In the other direction, one can always recover an isotropic flag $F_\bullet$ from a coset $gB$ by taking $V_i$ to be the span of the first $i$ columns of any matrix in $gB$. 

As in \cite{bcSects}, we define the \emph{base clan} of the clan $\gamma$ as the one obtained by replacing all signed natural numbers in $\tl{\gamma}$ by their signature. For example, ${+}1212{-}$ has base clan ${+}{-}{-}{+}{+}{-}$. Notice that the base clan of a skew-symmetric clan remains skew-symmetric, and consists only of $+$ and $-$ symbols. Base clans are in correspondence with closed $L$-orbits in $G/B$ (\cite{wyserThesis}).

The following lemma is the major step in identifying the sects.

\begin{nlem} \label{lem:clanorbit} Let $Q_\gamma$ and $Q_\tau$ be $L$-orbits in $G/B$ corresponding to skew symmetric \hbox{$(n,n)$-clans} $\gamma$ and $\tau$. Then $Q_\gamma$ and $Q_\tau$ lie in the same $P$-orbit of $G/B$ if and only if $\gamma$ and $\tau$ have the same base clan.
\end{nlem}

\begin{proof} First, we prove necessity by showing that if $\gamma$ has base clan $\tau$, then the representing flags for each $L$-orbit $Q_\gamma$ and $Q_\tau$ lie in the same $P$-orbit. Then all clans with base clan $\tau$ will lie in the same $P$-orbit. More precisely, we exhibit parabolic group elements that can iteratively transform the representing flag for $\tau$ into the representing flag for $\gamma$.
	
	Let $\gamma=c_1\cdots c_{2n}$ and $\tau=t_1\cdots t_{2n}$, and let $F_\bullet(\gamma)=\la \mb{v}_1,\dots, \mb{v}_{2n} \ra$ and $F_\bullet(\tau)=\la \mb{u}_1, \dots, \mb{u}_{2n} \ra$ be the corresponding flags constructed by Theorem \ref{thm:flag}. As each clan has the same signature at symbols of the same index, they have the same default permutation. Then, we have two kinds of cases to examine.
	
	\textbf{Case 1:} There is a pair of numbers $c_i=c_{2n+1-i}$ with $i\leq n$. Then we have
	\[(\mb{v}_i, \mb{v}_{2n+1-i}) = (\frac{1}{\sqrt{2}}(\mb{e}_r+\mb{e}_{2n+1-r}), \frac{1}{\sqrt{2}}( \mb{e}_r -\mb{e}_{2n+1-r})) \quad \textrm{and} \quad (\mb{u}_i,\mb{u}_{2n+1-i}) = (\mb{e}_r, -\mb{e}_{2n+1-r}) \]
	for some  $n+1 \leq r \leq 2n$. 
	
	Let $p^{r}$ denote the parabolic subgroup element\footnote{It is easy to check that the matrix of this map satisfies the defining condition \ref{eq:PL}. That the matrix of this map is symplectic according to the condition in \ref{eq:Sp} is routine (if lengthy) linear algebra. } defined by
	\begin{align*} 
	p^r:\ & \mb{e}_r \lmt \mb{e}_r+\mb{e}_{2n+1-r},  \\
	&\mb{e}_i\lmt \mb{e}_i \quad\text{ for } i\neq r.
	\end{align*}
	Note that each pair of vectors
	\[ (\mb{e}_r+\mb{e}_{2n+1-r}, \mb{e}_r - \mb{e}_{2n+1-r}) \quad \textrm{and} \quad (\mb{e}_r+\mb{e}_{2n+1-r}, -\mb{e}_{2n+1-r}) \]
	generates the same subspace, so it doesn't matter which one of the vectors $\frac{1}{\sqrt{2}}(\mb{e}_r-\mb{e}_{2n+1-r})$ or $\mb{e}_r$ appears as $\mb{v_j}$ in the flag $F_\bullet(\gamma)$. Then, the action of $p^r$ on $F_\bullet(\tau)$ has the effect of taking $\mb{u}_i$ to the span of $\mb{v}_i$, making the spans of $(p^r\cdot \mb{u}_i,p^r \cdot \mb{u}_{2n+1-i})$ and $(\mb{v}_i, \mb{v}_{2n+1-i})$ the same.
	
	\textbf{Case 2:} In this case, we have $c_i=c_j$ with $i<j\neq 2n+1-i$ so that by skew-symmetry, there is another pair of natural numbers $c_{2n+1-j}=c_{2n+1-i}$.  Without loss of generality, we can assume $i\leq n$. In this case, Theorem \ref{thm:flag} will yield a flag $F_\bullet(\gamma)$ with 
	\[(\mb{v}_i, \mb{v}_{j}) = (\frac{1}{\sqrt{2}}(\mb{e}_r+\mb{e}_{2n+1-s}),\frac{1}{\sqrt{2}}( \mb{e}_r -\mb{e}_{2n+1-s})) \]
	and
	\[(\mb{v}_{2n+1-j}, \mb{v}_{2n+1-i}) = (\pm\frac{1}{\sqrt{2}}(\mb{e}_s+\mb{e}_{2n+1-r}),\frac{1}{\sqrt{2}}( \mb{e}_s -\mb{e}_{2n+1-r})),\]
	where $n<r={\sigma}(i)$ and $n<s={\sigma}(2n+1-j)$. We also obtain the flag $F_\bullet(\tau)$ with 
	\[(\mb{u}_i, \mb{u}_{j})= (\mb{e}_r, \pm\mb{e}_{2n+1-s}) \]
	and
	\[(\mb{u}_{2n+1-j}, \mb{u}_{2n+1-i}) = (\mb{e}_s,-\mb{e}_{2n+1-r}).\]
	
	Then define a linear map by
	\begin{align*} 
	p^{r,s}:\ & \mb{e}_r \lmt \mb{e}_r+\mb{e}_{2n+1-s},  \\
	&\mb{e}_s \lmt \mb{e}_s+\mb{e}_{2n+1-r},  \\
	&\mb{e}_i\lmt \mb{e}_i \quad\text{ for } i\neq r,s.
	\end{align*}
	It is again routine to check that this map defines an element of $P$, so $p^{r,s} \cdot F_\bullet(\tau)$ is a flag in the same $P$-orbit. Similar to the previous case, this map takes $\mb{u}_i$ to the span of $\mb{v}_i$ and $\mb{u}_{2n+1-j}$ to the span of $\mb{v}_{2n+1-j}$, yielding pairs with the same span:
	\[ (p^{r,s}\cdot \mb{u}_i,p^{r,s} \cdot \mb{u}_{j})\quad \text{and} \quad (\mb{v}_i, \mb{v}_{j}),\]
	and
	\[ (p^{r,s}\cdot \mb{u}_{2n+1-j},p^{r,s} \cdot \mb{u}_{2n+1-i})\quad \text{and} \quad (\mb{v}_{2n+1-j}, \mb{v}_{2n+1-i}).\]
	
	Now, after we act on the flag $F_\bullet(\tau)$ by the appropriate element of the form $p^{r}$ for each pair of natural numbers $c_i=c_{2n+1-i}$ in $\gamma$, and the appropriate element $p^{r,s}$ for each family
	\[\{c_i=c_j, c_{2n+1-j}=c_{2n+1-i} \mid j\neq 2n+1-i\},\]
	then we obtain a flag which is an equivalent presentation of $F_\bullet(\gamma)$. Thus, $Q_\gamma$ is in the same $P$-orbit as $Q_\tau$. 
	
	To show the converse, it suffices to show that $L$-orbits corresponding to distinct base clans $\tau$ and $\lambda$ lie in distinct $P$-orbits. Let $\tau=t_1\cdots t_{2n}\neq\lambda=l_1 \cdots l_{2n}$ and let $F_\bullet(\tau)=\la \mb{u}_1,\dots, \mb{u}_{2n} \ra$ and $F_\bullet(\lambda)=\la \mb{w}_1,\dots, \mb{w}_{2n} \ra$ be flags constructed to represent each orbit using Theorem \ref{thm:flag}. Now let $i$ be the least index such that $t_i \neq l_i$. Without loss of generality, we may assume that $t_i=-$ and $l_i=+$, and that $i\leq n$. Then we have $\mb{u}_i=\mb{e}_r$ for some $n+1 \leq r \leq n$ and $\mb{w}_i = \mb{e}_q$ for some $1\leq q \leq n$. For these flags to be in the same $P$-orbit we would need to be able to carry $\mb{e}_q$ to a vector with non-zero $\mb{e}_r$ component via some $p\in P$. This would force a non-zero entry in the $(r,q)$-entry of the matrix of $p$, but since $p$ has a block-diagonal form with zero $(i,j)$-entry whenever $i>n$ and $j\leq n$, this is impossible. Thus, the $L$-orbits $Q_\tau$ and $Q_\lambda$ are in distinct $P$-orbits.
	
\end{proof}

Note that there is an isomorphism 
\begin{align}
\vfi: LgB &\lra B g^{-1}L \label{cosetswap} \\
lgb & \lmt (lgb)^{-1} \nonumber
\end{align}
for any such double coset. Thus, given a clan $\gamma$, we have a bijection between $L$-orbits in  $G/B$ and $B$-orbits in $G/L$ given by
\[Q_\gamma=Lg_\gamma B/B \lmt Lg_\gamma B \xra{\ \vfi \ } Bg_\gamma^{-1}L \lmt Bg_\gamma^{-1}L/L=:R_\gamma. \]
Pushing the consequences of Lemma \ref{lem:clanorbit} through this association and applying the projection $\pi:G/L\to G/P$, we see that the $B$-orbits of $G/L$ which project to the Schubert cell $B{g_\gamma}^{-1}P/P$ are exactly the set of $R_\tau$ where $\tau$ has the same base clan as $\gamma$. This yields the following, with $R_\gamma = Bg_\gamma^{-1}L/L$ as above.

\begin{nprop} \label{projbase}
	Let $R_\gamma$ and $R_\tau$ be $B$-orbits in $G/L$ corresponding to clans $\gamma$ and $\tau$, and let \hbox{$\pi: G/L \to G/P$} denote the canonical projection. Then $\pi(R_\gamma)= \pi(R_\tau)$ if and only if $\gamma$ and $\tau$ have the same base clan.
\end{nprop}

Having successfully grouped clans according to their base clans, we find it appropriate to also call the collection of clans with a base clan $\gamma$ the {\emph{sect}} of $\gamma$, denoted $\Sigma_\gamma$. By abuse of terminology, we will also use this term to describe collections of any of the objects ($B$-orbits, $L$-orbits, double cosets, etc.) parameterized by $\Sigma_\gamma$.

\subsection{Sects and Schubert Cells}\label{sasc}
As in \cite{billeyLak} \S 3.3, we fix $T$ as the maximal torus of diagonal matrices in $G$, and then the Weyl group $W=N_G(T)/T$ can be identified with the permutations in the subgroup of the symmetric group $\mc{S}_{2n}$ defined by the condition,
\[W\cong \{ \sigma\in \mc{S}_{2n} \mid \sigma(2n+1-i)=2n+1-\sigma(i) \text{ for all } i \}.\]
This happens to be isomorphic to the hyperoctahedral group $\mc{C}_n$. A Weyl group element given by a permutation $\sigma\in\mc{S}_{2n}$ can also be represented by a matrix \hbox{$w(\sigma) \in N_G(T)$} where the $i$\ts{th} column is $\pm\mb{e}_{\sigma(i)}$. Thus, $w(\sigma)$ will be almost the permutation matrix for $\sigma$, only with some $-1$ entries instead of $1$ entries to ensure that it is a symplectic group element.

The choice of $B$ as the Borel subgroup of upper-triangular matrices determines the set of positive roots as
\[\Phi^+= \{ (\ep_i\pm \ep_j), \mid 1\leq i < j \leq n \} \cup \{ 2\ep_i \mid 1\leq i \leq n \},\]
where $\pm\ep_i$ denotes the torus character that takes $\diag(t_1,\dots, t_{2n}) \mapsto t_i^{\pm 1}$.
We choose simple roots 
\[ \del=\{ \alpha_i=\ep_i-\ep_{i+1}\mid 1\leq i \leq n-1 \} \cup \{ \alpha_n=2\ep_{n} \},\]
so that our maximal parabolic subgroup $P$ corresponds to the the subset $\del \wo \{\alpha_n\}$. The Weyl group of $P$ is denoted $W_P$ and is isomorphic to a copy of  $\mc{S}_n$ with respect to which the minimal coset representatives of $W/W_P$ can be identified as the permutations $\sigma \in W$ which also satisfy
\[ \sigma(1)< \sigma(2) < \dots < \sigma(n).\]
The set of minimal coset representatives is denoted $W^P$. The values $\sigma(1),\dots, \sigma(n)$ become the indices of the standard basis vectors which appear (up to sign) in the first $n$ columns of the matrix $w(\sigma)$.\footnote{It can be arranged so that these columns are exactly standard basis vectors, as in Theorem \ref{thm:flag}.} 

Let $[ m ]:= \{1,\dots, m\}$. The coset $w(\sigma)P$ includes all of the $\mc{S}_n$-many permutations of the set 
\[ \{ \mb{e}_{\sigma (1)}, \dots , \mb{e}_{\sigma (n)} \} \] 
by the (right) action of the Levi subgroup on the columns of $w(\sigma)$.
So order does not actually matter, and the Schubert cell $Bw(\sigma)P/P$ can be understood as the Borel orbit of the isotropic subspace spanned by $\{ \mb{e}_{\sigma (1)}, \dots , \mb{e}_{\sigma (n)} \}$. 
Hence, we can make the identification 

\[
\text{Schubert cells } C_I  \llra \lt\{ \begin{aligned}  \text{Subsets } I \text{ of } [2n] \text{ such that }  \ab{I}=n & \\  \text{ and if } i\in I \text{ then } 2n+1-i\not\in I&\end{aligned}\rt\}.
\]
Since for each opposite pair of indices we choose just one for membership in the subset $I$, there are $2^n$ cells total. Each cell is also associated to a maximal isotropic subspace of the form
\[ V_I : = \spn( \{\mb{e}_i \mid i \in I \}) .\]

We can bijectively associate base clans to sets $I$ by defining $\gamma_I=c_1 \dots c_{2n}$ by
\begin{equation}\label{baseclanI} c_i =\begin{cases}
+ \quad \text{ if } i \in I \\
- \quad \text{ if } i \notin I .
\end{cases}\end{equation}
Note that under this assignment, after building the default flag, $F_\bullet({\gamma_I})=\la \mb{v}_1, \dots, \mb{v}_{2n}\ra$, it easy to check that $\spn(\{\mb{v}_1,\dots,\mb{v}_{n}\})=V_I$. In other words, $C_I = B g_{\gamma_I} P/P$. Finally, we have our promised result.

\begin{nthm} \label{decomp} Let $C_I$ be a Schubert cell of $G/P$, and $\pi: G/L \to G/P$ the natural projection. Associate to $I$ a base clan $\gamma_I$ as in equation \ref{baseclanI}, and denote the sect of $\gamma_I$ by $\Sigma_I$. If $R_\gamma$ denotes the $B$-orbit of $G/L$ associated to the clan $\gamma$, then
	\begin{equation}
	\pi^{-1}(C_I) = \bigsqcup_{\gamma \in \Sigma_I} R_\gamma.
	\end{equation}
\end{nthm}

\begin{proof} The pre-image of $C_I$ decomposes as a disjoint union of $B$-orbits $R_\gamma$ as a consequence of the fact that $\pi$ is $G$-equivariant, so in particular $\pi$ is $B$-equivariant. It remains to determine for which $\gamma$ we have $\pi(R_\gamma)=C_I$. Proposition \ref{projbase} tells us that $B$-orbits $R_\gamma$ and $R_\tau$ project to the same Schubert cell if and only if they are members of the same sect. Thus, we only have to prove that the $B$-orbit $R_{\gamma_I}$ projects to the Schubert cell $C_I$.
	
	We know that $\pi(R_{\gamma_I})=B g_{\gamma_I}^{-1} P /P$.  Observe that from the construction of $g_{\gamma_I}$, if we denote the $i$\ts{th} column of $g_{\gamma_I}^{-1}$ as $\mb{w}_i$, then 
	
	\[\mb{w}_i= \begin{cases} - \mb{e}_{\sigma(i)} &\textrm{  if  } i\leq n \text{  and  } c_i=-,\\
	\mb{e}_{\sigma(i)} & \text{  otherwise.}
	\end{cases} 
	\]
	Because this differs from the columns of $g_{\gamma_I}$ only possibly by a sign, the first $n$ columns span the same subspace, and in fact 
	\[ g_{\gamma_I}^{-1} P = g_{\gamma_I} P.\]
	
	This implies that 
	\[\pi(R_{\gamma_I})= B g_{\gamma_I}^{-1}P/ P = B  g_{\gamma_I}P/P = C_I.\]
\end{proof}

\section{The Big Sect} \label{sec:big}
Here, we are going to analyze the number of clans lying in the largest sect over $\Lambda(n)$, which we denote by $\epsilon_n$. The set of all such clans will be denoted by $\mc{E}_n$.

A clan lies in the largest sect if and only if it has a base clan in the following form $$\underbrace{{-} {-} \cdots {-} {-} {-} {-}}_{\text{first}\ n\text{-spots}} {+} {+}\cdots {+} {+} {+}. $$ This follows from the fact that this base clan is $\gamma_I$ where $V_I=\spn\{\mb{e}_{n+1}, \dots, \mb{e}_{2n}\}$ whose $B$-orbit $C_I$ is the dense Schubert cell in $\Lambda(n)$, just as it is in $\Gr(n, \C^{2n})$. Hence, a clan lies in the largest sect if and only if its only natural numbers are from $\Pi_{1,2}$ families and $\Pi_2$ pairs.

Let $\alpha$ denote the number of $\Pi_{1,2}$ families and $\beta$ the number of $\Pi_2$ pairs in an arbitrary clan $\gamma \in \mc{E}_n$. Then $2\alpha + \beta= k$ is 
the total number of pairs in $\gamma = c_1 \cdots c_n\ c_{n+1}\ \cdots c_{2n}$. 
To see in how many different ways these pairs of indices can be situated in such $\gamma$, we start by choosing $\beta$ spots from the first $n$ positions in $\gamma$. 
Obviously, this can be done in ${n\choose \beta}$ many different ways. Then, we count different ways of choosing $\alpha$ pairs within the $n-\beta$ spots remaining (among the first $n$) to place the families of type $\Pi_{1,2}$. We have ${{n-\beta}\choose {2\alpha}}$ spots to choose from, and $\frac{(2\alpha)!}{2^{\alpha} \alpha!}$ different ways to form pairs from these spots. This proves the following corollary.

\begin{ncor}\label{sectnumber}
The number of clans in the largest sect is given by 
\begin{align*}
\epsilon_n = \sum_{\substack{k=0 \\ 2\alpha + \beta = k}}^{n} {n \choose \beta}{{n-\beta}\choose {2\alpha}} \frac{(2\alpha)!}{2^{\alpha} \alpha!}, 
\hspace{.5cm} \text{ or equivalently, }  \hspace{.5cm}
\epsilon_n = \sum_{\substack{k=0 \\ 2\alpha + \beta = k}}^{n} {n \choose \beta} \frac{(n - \beta)!}{(n-\beta - 2\alpha)! 2^{\alpha} \alpha!}.
\end{align*} 

\end{ncor}

The first values of $\epsilon_n$, beginning with $n= 1$, are $2, 5, 14, 43, 142, 499, 1850, \dots$. This is in fact the number of self-inverse partial permutations, also known as partial involutions, (\href{http://oeis.org/A005425}{A005425}), as we show in the next section.

\subsection{Partial Involutions}
In this section, we will show that the set of all clans in the largest sect $\mc{E}_n$ is in one-to-one correspondence with the set of all partial involutions on $n$ elements.
We refer to~\cite{bagnocongruence} and Chapter 15 of \cite{CCA} for background theory.
Recall that a \emph{partial permutation} is a
map 
\[\pi : \{1, \dots, m\} \mapsto \{0, \dots, n\}\] satisfying the following rule: 
\begin{itemize}
	\item if $\pi(i) = \pi(j)$ and $\pi(i) \neq 0$, then $i = j$ for each $1 \leq i, j \leq m$.
\end{itemize}
A partial permutation matrix $\pi$ can be represented by an $m\x n$ matrix $(x_{ij})$, where $x_{ij}$ is 1 if and only if $\pi (i) = j$, and
0 otherwise. Note that under this convention we view our matrices as acting on vectors from the right. Partial permutations are sometimes called \emph{rook placements}, and in case $m =n$, they form a monoid under matrix multiplication called the \emph{rook monoid} and denoted $\mc{R}_n$. 

\begin{ndefn}
A partial involution on $n$ elements is a partial permutation which is represented by a symmetric $n \x n$ partial permutation matrix. The set of partial involutions on $n$ elements will be denoted by $\mc{P}_n$ and its cardinality by $P_n$.
\end{ndefn}

We will prove that $\epsilon_n = P_n$ by exhibiting an explicit bijection between
the clans in the largest sect and the partial involutions.
This will allow us to make use of a
 known recursive formula for the partial involutions to give a generating function for the sequence of $\epsilon_n$'s. A direct proof of the recurrence using the formula of Corollary~\ref{sectnumber} may be more difficult to achieve. Let us now state this recurrence relation.

\begin{nprop}\label{sectrec}
	Taking $P_0 = 1$, the numbers $P_n$ satisfy the recurrence relation
	\begin{eqnarray*}
	P_n = 2 P_{n-1} + (n-1) P_{n-2}
	\end{eqnarray*}
	for all $n > 1$.
\end{nprop}
 
In order to describe the bijection between $\mc{E}_n$ and $\mc{P}_n$, let $\pi$ be a partial involution with matrix $(x_{ij})$ for $1 \leq i, j \leq n$. We construct a clan $\gamma = c_1 c_2 \dots c_{2n}$ as follows:
\begin{itemize}
	\item[(i)] If $x_{ii} = 0$, then  $c_i = -$ and $c_{2n+1-i}=+$.  
	\item[(ii)] If $x_{ii} = 1$ for any $1 \leq i \leq n$, then  we have a transposition $(i, 2n+1-i)$ in the underlying involution of $\gamma$ which yields $c_i = c_{2n+1-i} = a$ for some $a \in \N$.
	\item[(iii)] If $x_{ij} = 1$ and $x_{ji}=1$ (by symmetry), then $(i, 2n+1-j)$ and $(j, 2n+1-i)$ are the corresponding transpositions in the underlying involution for our clan. Thus, $c_i=c_{2n+1-j} = a$ and $c_{j} = c_{2n+1-i}=b$ for some $a, b\in \N$.
\end{itemize}
For example, consider the matrix $(x_{ij}) = \begin{pmatrix}
	1 & 0&0\\
	0& 0& 1\\
	0& 1 & 0
	\end{pmatrix}$. Then since $x_{11} = 1$, we can take $c_1 = c_6 = 1$ in the corresponding clan. Moreover, since $x_{23}= 1$ and $x_{32} = 1$, we assign $c_2 = c_4 = 2$ and $c_3 =c_5 = 3$. Thus, the corresponding clan is $1 2 3 2 3 1$.

This algorithm can be reversed easily. Let us start with a clan $\gamma = c_1 c_2 \dots c_{2n}$ and define its associated partial involution matrix as the one with zeros everywhere except:
\begin{itemize}
	\item[(i)] If $c_{i} = c_{2n+1-i} = a\in \N$, then $x_{ii} = 1$.
	\item[(ii)] If $c_{i} = c_{2n+1-j} = a \in \N$ for $i\neq j$, then $x_{ij} =x_{ji} = 1$.
\end{itemize}

This algorithm and its reverse are clearly injective between partial involutions and clans from the largest sect, completing the bijection. 

\begin{nthm}\label{sectbijection}
	 Partial involutions on $n$ letters and skew-symmetric $(n, n)$-clans in the largest sect are in bijection.
\end{nthm}

\begin{ncor}
	Taking $\epsilon_0=1$, the number of clans in the largest sect satisfies the recurrence relation
	\begin{equation*}
	\epsilon_n =2\epsilon_{n-1} + (n-1)\epsilon_{n-2},
	\end{equation*}
	and has exponential generating function 
	\begin{equation*}
	\sum_{n=0}^\infty \epsilon_n\frac{x^n}{n!} = e^{2x+\frac{x^2}{2}}.
	\end{equation*}
\end{ncor}
\begin{proof}
	The recurrence follows from Theorem~\ref{sectbijection} and Proposition~\ref{sectrec}. The exponential generating function can be found by a straightforward calculation.
\end{proof}

Next we describe an order structure on partial involutions.  Let $B^-_m$ denote the group of invertible $m\x m$ lower triangular matrices, and $B_n$ the group of invertible upper triangular $n \x n$ matrices. $B^-_m \x B_n$ acts on the set of all $m \x n$ matrices with complex entries by $(a , b ) \cdot x = a x b^{-1}$. The double cosets of this action are indexed by the partial permutations.

 The group $B_n$ also acts on the set of symmetric $n \x n$ matrices $\bb{S}(n)$ by the \emph{congruence action}, which is defined by
\[ b \cdot s = (b^{-1})^t x b^{-1},\]
for $x\in \bb{S}(n)$.  The set of partial involutions $\mc{P}_n$ parametrize congruence $B_n$-orbits. In fact, for $\pi \in \mc{P}_n$, the congruence $B_n$-orbit $B_n\cdot \pi$ is exactly the intersection
$B_n^- \pi B_n \cap \bb{S}(n)$ (see \cite{szechtman}, Theorem 3.1). Moreover, the orbit closure order for congruence $B_n$-orbits is the same as the closure order on $B^-_n \x B_n$ double cosets restricted to the set of partial involutions (\cite{bagnocongruence}). This order can be described in terms of {rank-control matrices} associated to each partial permutation, which we define presently.
\begin{ndefn}\label{rkctl}
Let $\pi=(x_{ij})$ be an $m \times n$ matrix. For each $1 \leq k \leq m$ and $1 \leq l \leq n$, denote by $\pi_{k,l}$ the upper-left $k \times l$ submatrix of $\pi$. We denote by $R(\pi)$ the 
$m \times n$ matrix, whose $(k,l)$-entry is $ \text{rank} (\pi_{k,l})$, 
and call it the \emph{rank-control matrix} of $\pi$.
\end{ndefn}
\begin{nexap}
	If $\pi=	\begin{pmatrix}
	1 & 0&0\\
	0& 1& 0\\
	0&0&1
	\end{pmatrix}$, then $R(\pi) = \begin{pmatrix}
	1 & 1& 1\\
	1 & 2& 2\\
	1 & 2& 3
	\end{pmatrix}$.
\end{nexap}
We then define a partial order on partial permutations by comparing the individual entries of their rank-control matrices, which is to say
\[\pi \preceq \sigma \iff \rank(\pi_{k,l}) \leq \rank(\sigma_{k,l}) \quad \text{for all } k ,l . \]
Let ${C}_\pi$ and ${C}_\sigma$ denote the $B^-_m \x B_n$-orbits corresponding to partial permutations $\pi$ and $\sigma$. It follows from Theorem 15.31 of \cite{CCA} that
\[\pi \preceq \sigma \iff \ol{{C}_\pi} \sus \ol{{C}_\sigma}. \]
Thus, rank-control matrices capture all of the information of orbit closure relationships.
Corollary 5.4 of \cite{bagnocongruence} specializes this to say that the restriction of $\preceq$ to $\mc{P}_n$ exactly describes the closure order on congruence $B_n$-orbits as well. 

\subsection{A Conjectural Bruhat Order on $CI$ Clans}

In \cite{wyser2016bruhat}, Wyser gives a description of the closure order of $GL_p \x GL_q$-orbits in the flag variety of $GL_{p+q}$ in terms of statistics on the $(p,q)$-clans parametrizing these orbits.  We recall the statement and notation from that paper.

For any $(p,q)$-clan $\gamma=c_1 \cdots c_n$ and any $i$ or $i,j$ with $1\leq i <j \leq n$, we let
\begin{enumerate}[label=(\arabic*)]
\item $\gamma(i; +) =$ the total number of $+$ signs and pairs of equal natural numbers occurring among $c_1 \cdots c_i $;
\item $\gamma(i; -) =$ the total number of $-$ signs and pairs of equal natural numbers occurring among $c_1 \cdots c_i $;
 \item  $\gamma(i; j) =$ the number of pairs of equal natural numbers $c_s=c_t \in \N$ with $s \leq i < j< t$. 
 \end{enumerate}

\begin{nthm}[Theorem 1.2, \cite{wyser2016bruhat}] \label{BruhatClans}
Let $\gamma$, $\tau$ be $(p,q)$-clans, and let $Y_\gamma$, $Y_\tau$ be the corresponding $GL_p \x GL_q$-orbit closures in the flag variety of $GL_{p+q}$. Then $\gamma \leq \tau$ (meaning $Y_\gamma \sus Y_\tau$) if and only if the following three inequalities hold for all $i$, $j$:
\begin{enumerate}[label=(\arabic*)]
\item $\gamma(i;+)\geq \tau(i;+)$;
\item $\gamma(i;-)\geq \tau(i;-)$;
\item $\gamma(i;j)\leq \tau(i;j)$.
\end{enumerate}
\end{nthm}

Just as skew-symmetric $(n,n)$-clans can be viewed as a subset of all $(n,n)$-clans, the isotropic flag variety $X:=Sp_{2n}/B$ can be viewed as a subvariety of the flag variety $G'/B'$ for $G'=GL_{2n}$ and its subgroup $B'$ of upper triangular matrices. Indeed, the $L$-orbit in $Sp_{2n}/B$ corresponding to the skew-symmetric clan $\gamma$ is exactly the intersection $Q_\gamma'\cap X$,  where $Q_\gamma'$ is the $L':=GL_n\x GL_n$-orbit in $G'/B'$ corresponding to $\gamma$ (see \cite{wyserThesis}).

It is then natural to wonder if the (full) closure order on skew-symmetric $(n,n)$-clans is simply the restriction of the closure order on all $(n,n)$-clans. This is stated as Conjecture 3.6 in \cite{wyserKorbit}, where it is reported that the conjecture has be verified computationally up to $n=7$. The analogous conjecture is, in fact, made for two other classical symmetric pairs: $(SO_{2n+1}, S(O_{2p} \x O_{2q+1}))$ and $(Sp_n, Sp_p\x Sp_q)$. However, neither of the symmetric subgroups in these cases are Levi factors of a parabolic subgroup, so they hold less interest for our purposes. 

It is worth noting that the veracity of the conjecture for $(Sp_n, Sp_p\x Sp_q)$ was claimed in \cite{mcgovern2009pattern} on the basis of the recursive process for deducing the full closure order relations from the weak order. However, very little detail is provided, and it is also claimed in that paper that the same holds for $(SO_{2n}, GL_n)$, though Wyser points out that this fails with a specific example when $n=4$ (\cite{wyserKorbit}, p. 165).

A detailed proof of this conjecture for $(Sp_{2n}, GL_n)$ can likely be achieved via methods similar to \cite{wyser2016bruhat}, but we forego this at present. Instead, we will just assign a partial order on skew-symmetric $(n,n)$-clans by
\[\gamma \leq_C \tau \iff Y_\gamma \sus Y_\tau \quad (\text{so }\gamma\leq \tau), \]
with $Y_\gamma$ and $Y_\tau$ as before, and call $\leq_C$ the \emph{conjectural $CI$ Bruhat order}. With this in hand we can state and prove a final result, analogous to \cite{bcSects}, Theorem 1.7.

\begin{nthm} \label{poset}
As posets $(\mc{E}_n, \leq_C)\cong (\mc{P}_n, \preceq)$, that is the poset of the largest sect within skew-symmetric $(n,n)$-clans with the conjectural $CI$ Bruhat order is isomorphic to the poset of partial involutions ordered by congruence $B_n$-orbit closures. 
\end{nthm}
\begin{proof} Given a clan $\gamma \in \mc{E}_n$, let $\pi^\gamma$ denote the partial involution matrix obtained from $\gamma$ by the algorithm described before Theorem \ref{sectbijection}. We will show that the conditions of Theorem \ref{BruhatClans} on clans $\gamma$ and $\tau$ translate to all of the necessary conditions on the rank-control matrices of $\pi^\gamma$ and $\pi^\tau$, so that the orders are the same. This is achieved simply by identifying the statistics $\gamma(i;+)$, $\gamma(i; -)$ and $\gamma(i;j)$ as ranks of particular submatrices of the associated partial involution.  As in Definition \ref{rkctl}, let $\pi^\gamma_{k,l}$ denote the northwest $k \x l$ submatrix of $\pi^\gamma$, and $R(\pi^\gamma)$ the rank-control matrix of $\pi^\gamma$ so that the $(k,l)$-entry of $R(\pi^\gamma)$ is $\rank{\pi^\gamma_{k,l}}$. 

First we examine $\gamma(i;+)$. Because $\gamma=c_1 \cdots c_{2n}$ has  base clan ${-} {-} \dots {-} {-}{+}{+} \dots {+}{+}$, the value of $\gamma(i;+)$ is zero for all $i\leq n$, and then because each symbol among the last $n$ in $\gamma$ is either a $+$ or the second of a natural number pair, we find that $\gamma(i;+)=\tau(i;+)$ for all $i$ from $1$ to $2n$. Thus, this statistic contains no ordering information within the largest sect.

Next, consider $\gamma(i; -)$. As $i$ increases from 1 to $n$, so does $\gamma(i; -)$ unless there is a natural number at $c_i$. A natural number at $ c_i$ indicates a matrix entry of $1$ in the $i$\ts{th} row of $\pi^\gamma$ for $1\leq i \leq n$. As $i$ increases past $n$, $\gamma(i; -)$ increases the rest of the way up to $n$, and a natural number at $c_i$ indicates a matrix entry of 1 in the $(2n+1-i)$\ts{th} column. Then, 
\[ \gamma(i;-) = \begin{cases}
i - \rank( \pi^\gamma_{i,n}) & \text{for } i \leq n, \\
n- \rank ( \pi^\gamma_{n,2n+1-i}) & \text{for } i > n.
\end{cases}
\] 
Because the ranks of submatrices appear negatively above, $\gamma(i;-)\geq \tau(i; -)$ if and only if $\rank(\pi^\gamma_{i,n})\leq\rank(\pi^\tau_{i,n})$ and $\rank(\pi^\gamma_{n,i})\leq\rank(\pi^\tau_{n,i})$ for all $1\leq i \leq n$. This covers the rank conditions along the south and east borders of the matrices.

Next, consider $\gamma(i,j)$. By similar reasoning, it is not hard to see that 
\[\gamma(i;j)=
\begin{cases}
\rank (\pi^\gamma_{i,n} )& \text{for } i< j \leq n \\
\rank (\pi^\gamma_{i, 2n-j} )& \text{for } i\leq n < j \\
\rank (\pi^\gamma_{n,2n-j} )& \text{for } n<i< j .
\end{cases}
\]
The case where $i\leq n < j$ shows that $\gamma(i;j)\leq \tau(i;j)$ for all $1\leq i,j<n$ if and only $\rank(\pi^\gamma_{i,m}) \leq \rank(\pi^\tau_{i,m}) $ for all $1\leq i,m<n$. This covers the rank conditions everywhere else in the partial involution matrices, so we see that $\gamma\leq_C \tau$ if and only if $\pi^\gamma \preceq \pi^\tau$, completing the poset isomorphism.

\end{proof}

\begin{nrk} $(\mc{E}_n,\leq_C)$ is also a maximal upper order ideal of the poset of all skew-symmetric $(n,n)$-clans under $\leq_C$. Its minimal element is the base clan and the maximal element is $1 2 \cdots (n-1) n n  (n-1) \cdots 2 1 $, which corresponds to the dense $B$-orbit of $G/L$. This follows from arguments identical to the proof of Proposition 5.4 in \cite{bcSects}.
\end{nrk}

\begin{nrk} The proof of Theorem \ref{poset} can be adapted to show that the big sect of the type $AIII$ symmetric space ($GL_n/GL_p \x GL_q$) is isomorphic to the closure order on $B^-_p \x B_q$ double cosets of the $p \x q $ complex matrices, indexed by partial permutations. This was proved only for the case $p=q$ in \cite{bcSects}.
\end{nrk}


\vspace{1cm}
\textbf{Acknowledgements.}
We thank Mahir Bilen Can for many constructive suggestions. We are also grateful to the anonymous referee for their very careful reading and constructive suggestions which improved the quality of our paper.

\end{document}